\documentclass[reqno,11pt]{amsart}

\usepackage{tikz}
\usepackage{tikz-cd}
\usepackage{amssymb}
\usepackage{amsgen}
\usepackage{amsmath}
\usepackage{amsthm}
\usepackage{mathrsfs}
\usepackage{cite}
\usepackage{amsfonts}

\newcommand{\End}{\mathop{\mathrm{End}}}
\newcommand{\Spec}{\mathop{\mathrm{Spec}}}

\newcommand{\skel}[1]{^{(#1)}}

\newcommand{\dom}{\mathop{\boldsymbol d}}
\newcommand{\ran}{\mathop{\boldsymbol r}}

\renewcommand{\to}{\longrightarrow}

\newcommand{\supp}{\mathop{\mathrm{supp}}}

\newcommand{\inv}{^{-1}}
\newcommand{\p}{\varphi}
\newcommand{\pinv}{{\p \inv}}
\newcommand{\ov}[1]{\ensuremath{\overline {#1}}}
\newcommand{\til}[1]{\ensuremath{\widetilde {#1}}}

\usepackage{xcolor}

\newtheorem{Thm}{Theorem}[section]
\newtheorem{Prop}[Thm]{Proposition}

\newtheorem{Lemma}[Thm]{Lemma}
{\theoremstyle{definition}
\newtheorem{Def}[Thm]{Definition}}
{\theoremstyle{remark}
\newtheorem{Rmk}[Thm]{Remark}}
\newtheorem{Cor}[Thm]{Corollary}
{\theoremstyle{remark}
}
{\theoremstyle{remark}
}

{\theoremstyle{remark}
}
{\theoremstyle{remark}
}
{\theoremstyle{remark}
}

\numberwithin{equation}{section}

\title{Twists, crossed products and inverse semigroup cohomology}

\author{Benjamin Steinberg}
\address[B.~Steinberg]{%
    Department of Mathematics\\
    City College of New York\\
    Convent Avenue at 138th Street\\
    New York, New York 10031\\
    USA}
\email{bsteinberg@ccny.cuny.edu}

\thanks{The author was supported by a PSC CUNY grant.}
\date{January 23, 2021}


\begin{document}

\begin{abstract}
Twisted \'etale groupoid algebras have been studied recently in the algebraic setting by several authors in connection with an abstract theory of Cartan pairs of rings.  In this paper, we show that extensions of ample groupoids correspond in a precise manner to extensions of Boolean inverse semigroups. In particular, discrete twists over ample groupoids correspond to certain abelian extensions of Boolean inverse semigroups and we show that they are classified by Lausch's second cohomology group of an inverse semigroup.  The cohomology group structure corresponds to the Baer sum operation on twists.

We also define a novel notion of inverse semigroup crossed product, generalizing skew inverse semigroup rings, and prove that twisted Steinberg algebras of Hausdorff ample groupoids are instances of inverse semigroup crossed products.  The cocycle defining the crossed product is the same cocycle that classifies the twist in Lausch cohomology.
\end{abstract}

\maketitle

\section{Introduction}
Groupoid $C^*$-algebras have played an important role in the subject since Renault's seminal monograph~\cite{Renault}. An important development was a series of works, principally by Renault~\cite{Renault,renaultcartan} and Kumjian~\cite{Kumjiandiagonal}, in which it was shown that Cartan pairs of $C^*$-algebras correspond to twisted groupoid $C^*$-algebras for a certain class of \'etale groupoids.  This was used to good effect by Matui and Matsumoto in the study of symbolic dynamics via Cuntz-Krieger algebras~\cite{MatsMatu}.

An \'etale groupoid whose unit space is Hausdorff and has a basis of compact open sets is said to be ample~\cite{Paterson}.
The author introduced in~\cite{mygroupoidalgebra} algebras over any base commutative ring  associated to ample groupoids, nowadays called ``Steinberg algebras'', which are ring theoretic analogues of groupoid $C^*$-algebras.  They include group algebras, inverse semigroup algebras and Leavitt path algebras~\cite{LeavittBook,groupoidapproachleavitt}.  In recent years they have been the subject of investigation by a number of authors. See for example~\cite{Strongeffective,mygroupoidalgebra,operatorsimple1,operatorguys2,groupoidbundles,GroupoidMorita,Nekrashevychgpd,CarlsenSteinberg,ClarkPardoSteinberg,CenterLeavittGroupoid,Hazrat2017,Demeneghi,MyEffrosHahn,GonRoy2017a,mydiagonal,groupoidprimitive,groupoidprime,groupoidapproachleavitt,nonhausdorffsimple,simplicity}.

There were, in particular, a number of papers investigating when an ample groupoid can be recovered from the pair of its Steinberg algebra and the ``diagonal'' subalgebra of functions on the unit space, cf.~\cite{BCvH2017,reconstruct,mydiagonal,CarlsenSteinberg}.  This is similar in spirit to the work of Kumjian and Renault and may have motivated the authors of~\cite{twists} to introduce twisted Steinberg algebras of ample groupoids.  A theory of algebraic Cartan pairs was subsequently developed in~\cite{twistedreconstruction} that closely parallels the Renault-Kumjian theory from the $C^*$-algebra setting.

This paper arose from an attempt by the author to understand~\cite{twistedreconstruction} from the point of view of inverse semigroup theory.  It has been known for a long time that there is a close connection between ample groupoids and inverse semigroups~\cite{exelrecon,Renault,Paterson,Exel,LawsonLenz}.  The strategy of the author to reconstruct an ample groupoid from its Steinberg algebra and diagonal subalgebra in~\cite{mydiagonal} was as follows.  From a Steinberg algebra of an ample groupoid $\mathscr G$ and its diagonal subalgebra, one obtains an exact sequence of inverse semigroups $K\to S\to S/K$ where $S$ is the normalizer of the diagonal subalgebra and $K$ is the semigroup of diagonal normalizers (which is a normal inverse subsemigroup of $S$).  It turns out that if the original groupoid is effective, for example, then $S/K$ is isomorphic to the inverse semigroup of compact open bisections of $\mathscr G$.  Since $\mathscr G$ can be reconstructed from its inverse semigroup of compact open bisections via Exel's ultrafilter groupoid construction~\cite{Exel,exelrecon}, this shows that $\mathscr G$ is determined by its Steinberg algebra and diagonal subalgebra.

Therefore, it is natural to attempt to show that discrete twists over ample groupoids correspond to certain exact sequences of inverse semigroups.  Then one could reprove some of the results of~\cite{twistedreconstruction} via purely inverse semigroup theoretic means by showing that the exact sequence corresponding to the twist is equivalent to the exact sequence coming from the Cartan pair.  This paper sets the foundations for such an approach.

The papers of Bice~\cite{Bice1,Bice2} seem to be related to ours in that it wants to understand related algebras using inverse semigroup theory, but it does not apply the categorical and cohomological approach we do, nor the crossed product construction.

We begin the paper by giving a covariant equivalence of categories between ample groupoids and appropriately restricted functors and Boolean inverse semigroups and appropriately restricted homomorphisms.  Moreover, we show these functors send exact sequences of ample groupoids to exact sequences of inverse semigroups, and vice versa.  We then describe the exact sequences of inverse semigroups corresponding to twists over an ample groupoid by a discrete abelian group $A$ (written multiplicatively).  We show that $A$-twists over $\mathscr G$ correspond to extensions of the inverse semigroup  $\Gamma_c(\mathscr G)$ of compact open bisections of $\mathscr G$ by the commutative inverse semigroup $\til A$ of compactly supported locally constant functions $f\colon \mathscr G\skel 0\to A\cup \{0\}$.  Such extensions are then classified by Lausch's cohomology theory of inverse semigroups~\cite{lausch}.  In particular, we show that abelian group of  equivalence classes of $A$-twists under Baer sum is isomorphic to Lausch's second cohomology group $H^2(\Gamma_c(\mathscr G),\til A)$.  This is similar in spirit to how Lausch cohomology arose in the work on Cartan pairs of von Neumann algebras in~\cite{Donsig}.  We also show that the question of whether the twist admits a continuous global section can be interpreted inverse semigroup theoretically using ideas of~\cite{Donsig}.

In the final section of the paper, we introduce what we believe to be a new notion of inverse semigroup crossed product arising from an action of an inverse semigroup on a ring and a $2$-cocycle with respect to the action (a related, but different $C^*$-algebraic notion is in~\cite{ExelBuss}).  It generalizes the skew inverse semigroup ring construction for a large class of actions (those in~\cite{Skewasconv}), as well as group crossed products (with respect to an action).  After developing the basic properties of the crossed product (including a universal property), we show that twisted Steinberg algebras are inverse semigroup crossed products with respect to the same Lausch  $2$-cocycle that classifies the twist under our bijection between twists and cohomology classes.

\section{Groupoids and inverse semigroups}
Here we recall some basic notions about inverse semigroups and groupoids.

\subsection{Inverse semigroups}
An \emph{inverse semigroup} is a semigroup $S$ such that, for each $s\in S$, there is a unique element $s^*\in S$ with $ss^*s=s$ and $s^*ss^*=s^*$.  In an inverse semigroup, the idempotents commute and hence form a subsemigroup $E(S)$. Moreover, $e^*=e$ for any idempotent $e$.  Note that $ss^*,s^*s\in E(S)$ for any $s\in S$, as is $ses^*$ for any $e\in E(S)$.  We also observe that $(st)^*=t^*s^*$.  An element $s$ of a semigroup $S$ is (von Neumann) \emph{regular} if $s=ss's$ with $s'\in S$; a semigroup $S$ is (von Neumann) \emph{regular} if all its elements are regular.  It is well known that inverse semigroups are precisely the regular semigroups with commuting idempotents.

 There is a natural partial order on $S$ given by $s\leq t$ if $s=te$ for some idempotent $e\in E(S)$ or, equivalently, $s=ft$ for some $f\in E(S)$.  One can, in fact, take $e=s^*s$ and $f=ss^*$.  The natural partial order is compatible with multiplication and preserved (not reversed) by the involution. If $e,f\in E(S)$, then $ef$ is the meet of $e,f$ in the natural partial order and so $E(S)$ is a meet semilattice.  Any homomorphism $\p\colon S\to T$ of inverse semigroups automatically preserves the involution and order.  We say that $\p$ is \emph{idempotent separating} if $\p|_{E(S)}$ is injective.

A \emph{normal} inverse subsemigroup of an inverse semigroup $S$ is an inverse subsemigroup $K$ with $E(K)=E(S)$ and $sKs^*\subseteq K$ for all $s\in S$.   If $\p\colon S\to T$ is a homomorphism of inverse semigroups, then $\ker \p=\pinv(E(T))$ is a normal inverse subsemigroup and idempotent separating homomorphisms are ``determined'' by their kernels.   See~\cite{Lawson} for basics on inverse semigroup theory.

Many inverse semigroups have a zero element.  If $S$ is an inverse semigroup with zero, then we say that $s,t\in S$ are \emph{orthogonal} if $st^*=0=t^*s$.  A \emph{Boolean inverse semigroup} is an inverse semigroup $S$ with zero  such that $E(S)$ is a Boolean algebra, and which admits joins of orthogonal pairs of elements.  When we say that $E(S)$ is a Boolean algebra, we mean that it admits finite joins, and these joins distribute over meets, and it has relative complements (if $f\leq e$, then there exists $e\setminus f$ with $f(e\setminus f)=0$ and $e=f\vee (e\setminus f)$).  Note that if $s,t$ are orthogonal, then $(s\vee t)(s\vee t)^* = ss^*\vee tt^*$ and $(s\vee t)^*(s\vee t)=s^*s\vee t^*t$.  A homomorphism $\p\colon S\to T$ of Boolean inverse semigroups is called \emph{additive} if it preserves joins of orthogonal idempotents, in which case it preserves all finite joins existing in $S$. There are a number of other axiomatizations of Boolean inverse semigroups and we refer the reader to~\cite{wehrung} for more details.

If $X$ is a topological space, then the set $I_X$ of all homeomorphisms between open subsets of $X$ is an inverse semigroup under composition of partial functions.  An action of an inverse semigroup $S$ on a space $X$ by partial homeomorphisms is just a homomorphism $\alpha\colon S\to I_X$.  The action is \emph{non-degenerate} if $X$ is the union of the domains of the elements of $\alpha(S)$.

\subsection{Groupoids}
A \emph{groupoid} $\mathscr G$ is a small category in which each arrow is invertible.  We take the approach here, popular in analysis, of viewing $\mathscr G$ as a set with a partially defined multiplication and a totally defined inversion.  The objects of $\mathscr G$ are identified with the identity arrows (also called units) and the unit space is denoted $\mathscr G\skel 0$.  We use $\dom\colon \mathscr G\to \mathscr G\skel 0$ and $\ran\colon \mathscr G\to \mathscr G\skel 0$ for the domain and range maps, respectively.

A \emph{topological groupoid} is a groupoid endowed with a topology making the multiplication and inversion maps continuous.  As $\dom(g)=g\inv g$ and $\ran(g)=gg\inv$, the domain and range maps are also continuous where we give here $\mathscr G\skel 0$ the subspace topology.   When working with topological groupoids, we want our functors to be continuous.

A topological groupoid $\mathscr G$ is \emph{\'etale} if the domain map is a local homeomorphism.  This is equivalent to the range map being a local homeomorphism and also to the multiplication map being a local homeomorphism.  The unit space $\mathscr G\skel 0$ of an \'etale groupoid is open.  See~\cite{resendeetale} for details.  A (local) \emph{bisection} of an \'etale groupoid is an open subset $U\subseteq \mathscr G$ such that $\dom|_U$ and $\ran|_U$ are injective; some authors do not required bisections to be open and so we often say ``open bisection'' for emphasis.  The origin of the term ``bisection'' is that they correspond to subsets of $\mathscr G$ that are simultaneously  the image of a local section of $\dom$ and of $\ran$.  If $U,V$ are bisections, then so are $UV=\{gg'\mid g\in U, g'\in V\}$ and $U^*=\{g\inv\mid g\in U\}$. The bisections form an inverse semigroup with respect to this product operation with $U^*$ as the inverse of $U$. The idempotent bisections are the open subsets of $\mathscr G\skel 0$ and the natural partial order is just containment.

An \'etale groupoid $\mathscr G$ is called \emph{ample} (following Paterson~\cite{Paterson}) if $\mathscr G\skel 0$ is Hausdorff and has a basis of compact open sets.  In this case, the compact open bisections form a basis for the topology of $\mathscr G$.  The collection $\Gamma_c(\mathscr G)$ of compact open bisections of $\mathscr G$ is a Boolean inverse semigroup.  References for ample groupoids include~\cite{Paterson,Exel,mygroupoidalgebra}.  Ample groupoids arise in nature as groupoids of germs of actions of inverse semigroups on Hausdorff spaces with bases of compact open sets.

\section{Extensions of groupoids and inverse semigroups}

\subsection{An equivalence of categories}
In this subsection we show that the bijection between isomorphism classes of ample groupoids and Boolean inverse semigroups arising from the work of Exel~\cite{Exel,exelrecon} and Lawson and Lenz~\cite{LawsonLenz} can be turned into a covariant equivalence of categories if we restrict our morphisms suitably.  We will use this equivalence to show that discrete twists over an ample groupoid correspond to certain idempotent-separating extensions of inverse semigroups that are classified by Lausch cohomology~\cite{lausch}.

Let us call a continuous functor $\p\colon \mathscr G\to \mathscr H$ of  topological groupoids \emph{iso-unital} if $\p|_{\mathscr G\skel 0}\colon \mathscr G\skel 0\to \mathscr H\skel 0$ is a homeomorphism.  Note that if $\mathscr G$ and $\mathscr H$ are \'etale and $\p$ is open, then this is equivalent to $\p|_{\mathscr G\skel 0}\colon \mathscr G\skel 0\to \mathscr H\skel 0$ being bijective.  Notice that if $\p$ is iso-unital, then $\dom(\p(g)) = \dom(\p(g'))$ if and only if $\dom(g)=\dom(g')$, and dually $\ran(\p(g))=\ran(\p(g'))$ if and only if $\ran(g)=\ran(g')$.

\begin{Prop}\label{p:local.homeo}
Let $\p\colon \mathscr G\to \mathscr H$ be an iso-unital functor between \'etale groupoids.  Then the restriction of $\p$ to any bisection is injective.    Consequently, $\p$ is open if and only if $\p$ is a local homeomorphism.
\end{Prop}
\begin{proof}
Let $U\subseteq \mathscr G$ be a bisection.   If $g_1,g_2\in U$ and $\p(g_1)=\p(g_2)$, then since $\p$ is iso-unital, we must have $\dom(g_1)=\dom(g_2)$ and so $g_1=g_2$ as $U$ is a bisection.  Thus $\p|_U$ is injective.  Clearly, if $\p$ is a local homeomorphism, it is open.  Conversely, if $\p$ is open and $g\in \mathscr G$, then since $\mathscr G$ is \'etale, we can find an open bisection $U$ with $g\in U$. Then $\p|_U\colon U\to \p(U)$ is injective and open and hence a homeomorphism. Therefore, $\p$ is a local homeomorphism.
\end{proof}

The corresponding notion for inverse semigroups is the following.  A homomorphism $\p\colon S\to T$ of inverse semigroups is \emph{idempotent bijective} if $\p|_{E(S)}\colon E(S)\to E(T)$ is a bijection.  Note that if $S$ and $T$ are Boolean inverse semigroups, then an idempotent bijective homomorphism is clearly additive.

The following lemma combines two well-known facts, but we include a proof for completeness.
\begin{Lemma}\label{l:injective}
Let $\p\colon S\to T$ be an idempotent bijective homomorphism of inverse semigroups and $\psi\colon \mathscr G\to \mathscr H$ an iso-unital functor between topological groupoids.
\begin{enumerate}
\item  $\p$ is injective if and only if $\ker \p=\p\inv(E(T))=E(S)$.
\item $\psi$ is injective if and only if $\psi\inv(\mathscr H\skel 0)=\mathscr G\skel 0$.
\end{enumerate}
\end{Lemma}
\begin{proof}
If $\p$ is injective, then $\p(s)=e\in E(T)$ implies $\p(s^*s)=e^*e=e$ and so $s=s^*s\in E(S)$.  Conversely,  if $\p\inv(E(T))=E(S)$ and $\p(s_1)=\p(s_2)$, then since $\p$ is idempotent bijective, $s_1^*s_1=s_2^*s_2$.  Also $\p(s_1s_2^*)=\p(s_1s_1^*)\in E(T)$ and so $s_1s_2^*\in E(S)$.  Therefore, $s_1=s_1s_1^*s_1=s_1s_2^*s_2\leq s_2$ and dually $s_2\leq s_1$ and hence $s_1=s_2$.  Thus $\p$ is injective.

Similarly, if $\psi$ is injective, then $\psi(g)=x\in \mathscr H\skel 0$ implies $\psi(g)=\psi(\dom(g))$  and so $g=\dom(g)\in \mathscr G\skel 0$.  On the other hand, if $\psi\inv(\mathscr H\skel 0)=\mathscr G\skel 0$ and $\psi(g_1)=\psi(g_2)$, then $\dom(g_1)=\dom(g_2)$ and $\ran(g_1)=\ran(g_2)$, as $\psi$ is iso-unital, and $\psi(g_1g_2\inv)=\psi(g_1g_1\inv)\in \mathscr H\skel 0$ and so $g_1g_2\inv \in \mathscr G\skel 0$, whence $g_1=g_2$.
\end{proof}

To every ample groupoid $\mathscr G$, we can associate the Boolean inverse semigroup $\Gamma_c(\mathscr G)$.
The inverse construction is more difficult to describe.  If $S$ is a Boolean inverse semigroup, then $\Spec(E(S))$ denotes the \emph{Stone space} of the Boolean algebra $E(S)$.  It elements are the non-zero Boolean algebra homomorphisms (characters) $\chi\colon E\to \{0,1\}$.  A basis of compact open subsets for $\Spec(E(S))$ consists of the sets of the form \[D(e) = \{\chi\in \Spec(E(S))\mid \chi(e)=1\}\] and in fact $e\mapsto D(e)$ is an isomorphism of Boolean algebras between $E(S)$ and the Boolean algebra of compact open subsets of $\Spec(E(S))$.  There is a natural action of $S$ on $\Spec(E(S))$ by partial homeomorphisms. To each $s\in S$, there is a homeomorphism $\beta_s\colon D(s^*s)\to D(ss^*)$ given by $\beta_s(\chi)(e) = \chi(s^*es)$ and the assignment $s\mapsto \beta_s$ is a non-degenerate action of $S$ by partial homeomorphisms. We often write $s\chi$ instead of $\beta_s(\chi)$.   Put $\mathcal G(S)=S\ltimes \Spec(E(S))$, the groupoid of germs.  Recall that \[\mathcal G(S) = (S\times \Spec(B))/{\sim}\] where $(s,\chi)\sim (t,\lambda)$ if and only if $\chi=\lambda$ and there exists $u\leq s,t$ with $\chi(u^*u)=1$.  We write $[s,\chi]$ for the equivalence class of $(s,\chi)$.  The groupoid multiplication is given by defining $[s,\chi][t,\lambda]$ if and only if $\chi = t\lambda$, in which case the product is $[st,\lambda]$.  A basis for the topology is given by the compact open bisections \[D(s) = \{[s,\chi]\mid \chi\in D(s^*s)\}\] and this topology makes $\mathcal G(S)$ into an ample groupoid. The unit space $\mathcal G(S)\skel 0$ can be identified homeomorphically with $\Spec(E(S))$ via $\chi\mapsto [e,\chi]$ where $e\in E(S)$ with $\chi(e)=1$.  Under this identification, $\dom([s,\chi])=\chi$ and $\ran([s,\chi]) = s\chi$.  Inversion is given by $[s,\chi]^{-1}=[s^*,s\chi]$.  See~\cite{Exel,Paterson,mygroupoidalgebra} for more on groupoids of germs. (In~\cite{LawsonLenz}, they work with a groupoid of ultrafilters on the poset $S$, but it is well known and easy to see that this groupoid is isomorphic to $\mathcal G(S)$.)

The following theorem combines work of Exel~\cite{Exel,exelrecon}, Lawson and Lenz~\cite{LawsonLenz}.

\begin{Thm}[Exel-Lawson-Lenz]\label{t:duality}
Let $\mathscr H$ be an ample groupoid and $S$ a Boolean inverse semigroup.
\begin{enumerate}
\item  There is an isomorphism $\eta_{\mathscr H}\colon \mathscr H\to \mathcal G(\Gamma_c(\mathscr H))$ given by $\eta_{\mathscr H}(h) = [U,\chi_{\dom(h)}]$ where $U$ is any compact open bisection containing $h$ and $\chi_x$ is the characteristic function for the set of compact open subsets of $\mathscr H\skel 0$ containing $x\in \mathscr H\skel 0$.
\item There is an isomorphism $\varepsilon_S\colon S\to \Gamma_c(\mathcal G(S))$ given by $\varepsilon_S(s) = D(s)$.
\end{enumerate}
\end{Thm}

We now show that the constructions $\mathscr H\mapsto \Gamma_c(\mathscr H)$ and $S\mapsto \mathcal G(S)$ can be promoted to inverse equivalences between the category of ample groupoids with open iso-unital functors and the category of Boolean inverse semigroups with idempotent bijective homomorphisms.

\begin{Thm}\label{t:equiv.cats}
The category of ample groupoids with open iso-unital functors is equivalent to the category of Boolean inverse semigroups with idempotent bijective homomorphisms.  More precisely, $\Gamma_c$ is a functor from the category of ample groupoids with open iso-unital functors to the category of Boolean inverse semigroups with idempotent bijective homomorphisms, where  if $\p\colon \mathscr G\to \mathscr H$ is an open iso-unital functor, then $\Gamma_c(\p)\colon \Gamma_c(\mathscr G)\to \Gamma_c(\mathscr H)$ is given by $\Gamma_c(\p)(U)=\p(U)$.  Moreover, the groupoid of germs construction $S\mapsto \mathcal G(S)$ provides a quasi-inverse functor where if $\p\colon S\to T$ is idempotent bijective, then $\mathcal G(\p)([s,\chi]) = [\p(s), \chi\circ (\p|_{E(S)})^{-1}]$.
\end{Thm}
\begin{proof}
We begin by showing that $\Gamma_c$ is a functor.  Let $\p\colon \mathscr G\to \mathscr H$ be an open iso-unital functor.
Let $U\in \Gamma_c(\mathscr G)$.  Then $\p(U)$ is compact open; we just need to show that it is a bisection.  Let $g_1,g_2\in U$ with $\dom(\p(g_1))=\dom(\p(g_2))$.  Then  $\dom(g_1)=\dom(g_2)$, as $\p$ is iso-unital, and so $g_1=g_2$ because $U$ is a bisection.  Thus $\dom|_{\p(U)}$ is injective and the same argument applies to $\ran|_{\p(U)}$.  We conclude that $\p(U)\in \Gamma_c(\mathscr G)$.  Next we check that $\p(UV)=\p(U)\p(V)$ for $U,V\in \Gamma_c(\mathscr G)$.  If $g\in UV$, then $g=g_1g_2$ with $g_1\in U$ and $g_2\in V$, and so  $\p(g)=\p(g_1)\p(g_2)\in \p(U)\p(V)$.  Thus $\p(UV)\subseteq \p(U)\p(V)$.  Conversely, suppose that $h\in \p(U)\p(V)$.  So $h=\p(g_1)\p(g_2)$ with $g_1\in U$ and $g_2\in V$. Since $\p$ is iso-unital, this implies $\dom(g_1)=\ran(g_2)$ and so $h=\p(g_1g_2)$ with $g_1g_2\in UV$.  Thus $\p(UV)=\p(U)\p(V)$ as required.
Observe that since $\p|_{\mathscr G\skel 0}\colon \mathscr G\skel 0\to \mathscr H\skel 0$ is a homeomorphism and $E(\Gamma_c(\mathscr G))$ is the set of compact open subsets of $\mathscr G\skel 0$ and $E(\Gamma_c(\mathscr H))$ is the set of compact open subsets of $\mathscr H\skel 0$, it follows that $\Gamma_c(\p)$ is idempotent bijective.
It is clear from the definition that $\Gamma_c$ is functorial.

Next we show that if $\p\colon S\to T$ is an idempotent bijective homomorphism of Boolean inverse semigroups, then $\mathcal G(\p)$ is a well-defined, open iso-unital functor.   To ease notation, we write $\p_*$ for $\mathcal G(\p)$.  To check that $\p_*$ is well defined, suppose that $u\leq s,s'$ with $\chi(u^*u)=1$.  Then $\p(u)\leq \p(s),\p(s')$ and $\chi\circ (\p_{E(S)})^{-1}(\p(u)^*\p(u)) = \chi(u^*u)=1$.  Thus $[\p(s),\chi\circ (\p_{E(S)})^{-1}]=[\p(s'),\chi\circ (\p_{E(S)})^{-1}]$.  It's immediate from $\p$ being a homomorphism that $\p_*$ is a functor.  It is clearly bijective on unit spaces as $\chi\mapsto \chi\circ (\p_{E(S)})^{-1}$ is a bijection of Stone spaces.  It remains to show that $\p_*$ is open and continuous.  For continuity, we claim that $\p_*\inv(D(t)) = \bigcup_{\p(s)\leq t} D(s)$.  Indeed, if $\p(s)\leq t$ and $\chi(s^*s)=1$, then $\p_*([s,\chi]) = [\p(s),\chi\circ (\p_{E(S)})\inv] = [t,\chi\circ (\p|_{E(S)})\inv]\in D(t)$.  Conversely, if $\p_\ast([s,\chi])=[\p(s),\chi\circ (\p_{E(S)})^{-1}]\in D(t)$, then there exists $u\leq \p(s),t$ with $\chi\circ (\p_{E(S)})^{-1}(u^*u)=1$.  Put $e=(\p_{E(S)})^{-1}(u^*u)\leq (\p_{E(S)})^{-1}(\p(s)^*\p(s))=s^*s$ and let $s_0=se$.  Then $\p(s_0) = \p(s)\p(e)= \p(s)u^*u=u\leq t$ and $\chi(s_0^*s_0) = \chi(e) = \chi\circ (\p_{E(S)})^{-1}(u^*u)=1$.  Thus $[s,\chi] = [s_0,\chi]\in D(s_0)$ with $\p(s_0)\leq t$.   To show that $\p_*$ is open, we claim that $\p_*(D(s)) = D(\p(s))$.  From what we have just shown, $\p_*(D(s))\subseteq D(\p(s))$.  Suppose that $[\p(s),\lambda]\in D(\p(s))$. Let $\chi  =\lambda\circ \p|_{E(S)}$.  Then $\chi(s^*s) = \lambda(\p(s^*s)) = 1$ and so $[s,\chi]\in D(s)$.  Moreover, $\p_*([s,\chi]) = [\p(s), \chi\circ (\p|_{E(S)})^{-1}] = [\p(s),\lambda]$, as required.  It is clear that if $\p,\psi$ are composable idempotent bijective homomorphisms, then $\mathcal G(\p\psi) = \mathcal G(\p)\circ\mathcal G(\psi)$.

It now remains to show that $\eta$ and $\varepsilon$ from Theorem~\ref{t:duality} are natural isomorphisms.  Let $\p\colon \mathscr G\to \mathscr H$ be an open iso-unital functor.  Let $g\in \mathscr G$ and $U$ be a compact open bisection containing $g$.  Note that $\p(U)$ is a compact open bisection containing $\p(g)$ and $\mathcal G(\Gamma_c(\p))(\eta_{\mathscr G}(g)) = \mathcal G(\Gamma_c(\p))[U,\chi_{\dom(g)}] = [\p(U),\chi_{\dom(\p(g))}] = \eta_{\mathscr H}(\p(g))$.  Next we check that $\varepsilon$ is a natural transformation. If $\p\colon S\to T$ is an idempotent bijective homomorphism of inverse semigroups, then $\mathcal G(\p)(\varepsilon_S(s)) = \mathcal G(\p)(D(s)) = D(\p(s))=\varepsilon_T(\p(s))$ by the computation of the previous paragraph. This completes the proof.
\end{proof}

We next want to show that our functors preserve injective and surjective maps.  Presumably, there is a categorical description of these morphisms in our categories but we give a direct proof.

\begin{Prop}\label{p:preserve.epimono}
The functors $\Gamma_c$ and $\mathcal G$ preserve injective and surjective morphisms.
\end{Prop}
\begin{proof}
Let $\p\colon \mathscr G\to \mathscr H$ be an open iso-unital functor.  Suppose first that $\p$ is injective. Then, for $U\in\Gamma_c(\mathscr G)$, we have that $\p(U)\in E(\Gamma_c(\mathscr H))$ implies $\p(U)\subseteq \mathscr H\skel 0$ and hence $U\subseteq \mathscr G\skel 0$ by Lemma~\ref{l:injective}.  Thus $U\in E(\Gamma_c(\mathscr G))$ and so we conclude $\Gamma_c(\p)$ is injective by another application of Lemma~\ref{l:injective}.

Next suppose that $\p$ is a surjective  and let $V\in \Gamma_c(\mathscr H)$.  Note that $V$ is compact Hausdorff, being homeomorphic to the compact Hausdorff space $\dom(V)$. We claim that, for each $h\in V$, there is a compact open neighborhood $V_h\subseteq V$ of $h$ such that there is a continuous section $s\colon V_h\to \mathscr G$ of $\p$.  Choose $g\in \mathscr G$ with $\p(g)=h$.  Then since $\pinv(V)$ is open and the compact open bisections form a basis for the topology of $\mathscr G$, as $\mathscr G$ is ample, we can find a compact open bisection $U_g$ with $g\in U_g$ and $\p(U_g)\subseteq V$.  Then $V_h=\p(U_g)$ is compact open containing $h$ (since $\p$ is open).  Moreover, since $\p$ is iso-unital and $U_g$ is a bisection, $\p|_{U_g}$ is injective by Proposition~\ref{p:local.homeo}. Then $s=(\p|_{U_g})\inv$ is our desired section, as $\p$ is open.  Since $V$ is compact, we can cover $V$ by finitely many compact open subsets $V_1,\ldots, V_n$ such that there is a continuous section $s_i\colon V_i\to \mathscr G$ of $\p$.  Since $V$ is compact Hausdorff, its compact open subsets form a Boolean algebra.  The finitely many compact open subsets $V_1,\ldots, V_n$ of $V$ generate a finite Boolean algebra with maximum element $V=\bigcup_{i=1}^nV_i$.  Let $W_1,\ldots, W_k$ be the atoms of this Boolean algebra.  Then the $W_i$ are pairwise disjoint.  Also since $W_i=V\cap W_i=(V_1\cap W_i)\cup \cdots \cup (V_n\cap W_i)$, we see using $W_i$ is an atom that $W_i\subseteq V_j$ for some $j$.  Hence there is a  continuous section $t_i=s_j|_{W_i}\colon W_i\to \mathscr G$ of $\p$.  Since the maximum of a finite Boolean algebra is the join of its atoms, we deduce that $V=W_1\cup\cdots\cup W_k$ and this is a disjoint union.  Hence we can define a continuous section $s\colon V\to \mathscr G$ by putting $s|_{W_i}=t_i$ for $i=1,\ldots, k$.  Note that since $\p$ is a local homeomorphism by Proposition~\ref{p:local.homeo}, every continuous section of $\p$ defined on an open subset of $\mathscr H$ is an open mapping. Put $U=s(V)$ and note that $U$ is compact open and $\p(U)=V$.  Moreover, $\p|_U$ is injective as $\p\circ s=1_V$.    We claim that $U\in \Gamma_c(\mathscr G)$.    If $g,g'\in U$ with $\dom(g)=\dom(g')$, then $\dom(\p(g))=\dom(\p(g'))$ and so $\p(g)=\p(g')$ since $V$ is a bisection.  Therefore, $g=g'$ as $\p|_U$ is injective.  Similarly,  $\ran|_U$ is injective and so $U\in \Gamma_c(\mathscr G)$.  Thus $\Gamma_c(\p)$ is surjective.

Suppose now that $\p\colon S\to T$ is an idempotent bijective homomorphism of Boolean inverse semigroups.  First assume that $\p$ is injective and suppose that $[\p(s),\chi\circ (\p|_{E(S)})\inv]=\mathcal G(\p)([s,\chi])\in \mathcal G(T)\skel 0$.  Then there is an idempotent $f\in E(T)$ with $f\leq \p(s)$ and $\chi\circ (\p|_{E(S)})\inv (f)=1$.  Putting $e=(\p|_{E(S)})\inv(f)$, we have that $\chi(e)=1$ and $\p(se)=\p(s)f=f$.  Since $\pinv(E(T))=E(S)$ by Lemma~\ref{l:injective}, we have that $se\in E(S)$ and $se\leq s$.  Moreover, $\chi(es^*se) =1$ and so $[s,\chi]=[se,\chi]\in \mathcal G\skel 0$.  Therefore, $\mathcal G(\p)$ is injective by Lemma~\ref{l:injective}.  Next assume that $\p$ is surjective and let $[t,\chi]\in \mathcal G(T)$.  Then $t=\p(s)$ for some $s\in S$ and $\chi\circ \p|_{E(S)}(s^*s)=\chi(t^*t)=1$.  Thus $[s,\chi\circ \p|_{E(S)}]\in \mathcal G(S)$ and  we have $\mathcal G(\p)([s,\chi\circ \p|_{E(S)}]) = [\p(s),\chi]=[t,\chi]$.  Therefore, $\mathcal G(\p)$ is surjective.
\end{proof}

\begin{Rmk}\label{rmk:compact.section}
The proof of Proposition~\ref{p:preserve.epimono} shows that if $\p\colon \mathscr G\to \mathscr H$ is a surjective open iso-unital functor between ample groupoids, then a continuous local section of $\p$ can be defined on any compact open bisection of $\mathscr H$.
\end{Rmk}

Note that since $\Gamma_c$ and $\mathcal G$ are equivalences of categories and isomorphisms in these categories are bijective, it follows that $\p\colon \mathscr G\to \mathscr H$ is injective (respectively, surjective) if and only if $\Gamma_c(\p)\colon \Gamma_c(\mathscr G)\to \Gamma_c(\mathscr H)$ is injective (respectively, surjective) and $\p\colon S\to T$ is injective (respectively, surjective) if and only if $\mathcal G(\p)$ is injective (respectively, surjective).

The question of whether a surjective open iso-unital functor admits a continuous section preserving unit spaces (but not necessarily a functor) is important when determining whether a twist comes from a groupoid $2$-cocycle~\cite{twists}.  We show that in the ample setting, this is equivalent to the question considered in~\cite{Donsig} of when a surjective idempotent bijective inverse semigroup homomorphism admits an order-preserving and idempotent-preserving section.

The following is a minor variation of~\cite[Lemma~4.2]{Donsig}.
\begin{Lemma}\label{l:donsig}
Let $\p\colon S\to T$ be an idempotent bijective surjective homomorphism of inverse semigroups.  Then the following are equivalent.
\begin{enumerate}
\item There is an order-preserving map $j\colon T\to S$ with $j(E(T))\subseteq E(S)$ (i.e., $j$ is idempotent preserving) and $\p\circ j=1_T$.
\item There is a map $j\colon T\to S$ with $\p\circ j=1_T$ and $j(te)=j(t)j(e)$ for all $t\in T$ and $e\in E(T)$.
\item There is a map $j\colon T\to S$ with $\p\circ j=1_T$ and  $j(et)=j(e)j(t)$ for all $t\in T$ and $e\in E(T)$.
\end{enumerate}
\end{Lemma}
\begin{proof}
We prove the equivalence of the first and second items as the equivalence of the first and third items is dual.  Suppose first that $j$ as in (2) exists.  Then if $e\in E(T)$, we have that $j(e) = j(ee)=j(e)j(e)$ and so $j$ preserves idempotents.   Also, if $t_1\leq t_2$, we may write $t_1=t_2e$ with $e\in E(T)$.  Therefore, $j(t_1)=j(t_2e) = j(t_2)j(e)\leq j(t_2)$, as $j$ preserves idempotents, and so $j$ is order preserving.  Thus (2) implies (1).

Suppose that $j$ is as in (1).   Let $t\in T$ and $e\in E(T)$.   First we claim that $j(t)^*j(t)j(e)=j(te)^*j(te)$.  Indeed, applying $\p$ to both sides yields $t^*te$. Consequently, $j(t)^*j(t)j(e)=   j(te)^*j(te)$ as $j(e)\in E(S)$ and $\p$ is idempotent bijective.  Since $te\leq t$, the fact that $j$ is order preserving implies that $j(te)\leq j(t)$ and so
\[j(te)=j(t)j(te)^*j(te) = j(t)j(t)^*j(t)j(e)=j(t)j(e)\] as required.
\end{proof}

\begin{Prop}\label{p:section}
Let $\p\colon \mathscr G\to \mathscr H$ be a surjective open iso-unital functor between ample groupoids.  There there is a continuous mapping $f\colon \mathscr H\to \mathscr G$ with $\p\circ f=1_{\mathscr H}$ and $f(\mathscr H\skel 0)\subseteq \mathscr G\skel 0$ if and only if there is an order-preserving and idempotent-preserving mapping $j\colon \Gamma_c(\mathscr H)\to \Gamma_c(\mathscr G)$ with $\Gamma_c(\p)\circ j=1_{\Gamma_c(\mathscr H)}$.
\end{Prop}
\begin{proof}
Suppose first that $f$ exists.  Then since $\p$ is a local homeomorphism by Proposition~\ref{p:local.homeo}, the mapping $f$ is open.  Let $U\in \Gamma_c(\mathscr H)$ and put $j(U)=f(U)$.  Then $j(U)$ is compact open.  It is also a bisection since if $f(h),f(h')\in f(U)=j(U)$ with $\dom(f(h))=\dom(f(h'))$ and $h,h'\in U$, then $\dom(h)=\dom(h')$ because $\p$ is a functor and $\p\circ f=1_{\mathscr H}$.  Thus $h=h'$ because $U$ is a bisection.  Similarly, $\ran|_{j(U)}$ is injective and so $j(U)\in \Gamma_c(\mathscr G)$.  Obviously, $j$ is order preserving and preserves idempotents since $f(\mathscr H\skel 0)\subseteq \mathscr G\skel 0$.

For the converse, it suffices by Theorem~\ref{t:equiv.cats} to show that if $\p\colon S\to T$ is a surjective idempotent bijective homomorphism of Boolean inverse semigroups admitting an order-preserving section $j\colon T\to S$ with $j(E(T))\subseteq E(S)$, then $\mathcal G(\p)\colon \mathcal G(S)\to \mathcal G(T)$ admits a continuous section $f$ preserving units.  Put $f([t,\chi]) =[j(t),\chi\circ \p]$.  This is well defined since if $[t,\chi]=[t',\chi]$, then we can find $u\leq t,t'$ with $\chi(u^*u)=1$ and so $j(u)\leq j(t),j(t')$ as $j$ is order preserving and $\chi(\p(j(u)^*j(u))) = \chi(u^*u)=1$.  Thus $[j(t),\chi\circ \p]=[j(t'),\chi\circ \p]$.  Trivially, $\mathcal G(\p)(f([t,\chi]))= \mathcal G(\p)([j(t),\chi\circ \p])= [\p(j(t)),\chi]=[t,\chi]$.  Also if $[e,\chi]\in \mathcal G(T)\skel 0$ with $e\in E(T)$, then $f([e,\chi]) = [j(e),\chi\circ \p]\in \mathcal G(S)\skel 0$ since $j(e)\in E(S)$.

For continuity, we claim that $f\inv(D(s)) = \bigcup_{j(t)\leq s}D(t)$.  Indeed, if $j(t)\leq s$, then $f([t,\chi]) = [j(t),\chi\circ\p]=[s,\chi\circ \p]\in D(s)$.  Conversely, if $[j(t),\chi\circ \p]=f([t,\chi])\in D(s)$, then there exists $u\leq j(t),s$ with $\chi(\p(u^*u))=1$.  Put $e=u^*u$ and note that $j(t)e=u=se$.
Then $\p(u)\leq \p(j(t))=t$ and $\chi(\p(u)^*\p(u)))=1$ and so $[t,\chi]=[\p(u),\chi]\in D(\p(u))$.  Since $\p$ is idempotent bijective and $j$ is idempotent-preserving, we must have that $j(\p(e))=e$. Thus $t\p(e)=\p(j(t)u^*u) = \p(u)$, and so $j(\p(u))=j(t\p(e)) = j(t)j(\p(e))=j(t)e=se\leq s$ by Lemma~\ref{l:donsig}.  This completes the proof.
\end{proof}

It is shown in~\cite[Proposition~4.6]{Donsig} that if $\p\colon S\to T$ is an idempotent bijective surjective homomorphism of Boolean inverse monoids whose idempotents form a complete Boolean algebra, then an order-preserving and idempotent-preserving section always exists.  In the language of ample groupoids, these conditions mean that the unit spaces of $\mathcal G(S)$ and $\mathcal G(T)$ are Stonean (compact and extremally disconnected).

In~\cite{twists} it was shown that twists over second countable Hausdorff ample groupoids admit a unit-preserving global section; in~\cite[Lemma~2.4]{twistedreconstruction} and the remark thereafter, it was observed that the same proof works for paracompact Hausdorff ample groupoids.  The following proposition generalizes the argument to open iso-unital functors.

\begin{Prop}\label{p:section.sigma.compact}
Let $\p\colon \mathscr G\to \mathscr H$ be a surjective open iso-unital functor between ample groupoids.  If $\mathscr H$ is Hausdorff and $\mathscr H\setminus \mathscr H\skel 0$ is paracompact, then there is a continuous section $s\colon \mathscr H\to \mathscr G$ with $s(\mathscr H\skel 0)\subseteq \mathscr G\skel 0$.
\end{Prop}
\begin{proof}
Since $\mathscr H$ is Hausdorff, we have that $\mathscr H\skel 0$ is clopen in $\mathscr H$.  Therefore, $\mathscr H'=\mathscr H\setminus \mathscr H\skel 0$ is clopen. First we show that every $\sigma$-compact clopen subset $X$ of $\mathscr H'$ admits a continuous section $s_X\colon X\to \mathscr G$. Since $X$ is a countable union of compact sets and each of these compact sets can be covered by finitely many compact open bisections of $\mathscr H$ contained in $X$, as $\mathscr H$ is ample and $X$ is clopen, we have that $X=\bigcup_{n=1}^\infty U_n$ with the $U_n$ compact open bisections of $\mathscr H$.
Put $V_1=U_1$ and $V_n = U_n\setminus \bigcup_{i=1}^{n-1}V_i$.  Then since $\mathscr H$ is Hausdorff, each $V_n$ is a compact open bisection and by construction $\bigcup_{n=1}^\infty V_n=\bigcup_{n=1}^\infty U_n=X$.  Moreover, the $V_n$ are pairwise disjoint.  By Remark~\ref{rmk:compact.section}, there is a continuous section $s_n\colon V_n\to \mathscr G$ of $\p$ and we can then define $s_X$ by $s_X|_{V_n}=s_n$ for $n\geq 1$.

Next we use the well-known fact that a locally compact Hausdorff space is paracompact if and only if it can be partitioned into clopen $\sigma$-compact subspaces to write $\mathscr H'=\coprod_{\alpha\in A}X_{\alpha}$ with $X_{\alpha}$ clopen and $\sigma$-compact.  We now define $s\colon \mathscr H\to \mathscr G$ by $s|_{\mathscr H\skel 0} = (\p|_{\mathscr G\skel 0})\inv$ and $s|_{X_{\alpha}}=s_{X_{\alpha}}$.
This completes the proof.
\end{proof}

\subsection{Extensions of inverse semigroups}
By an \emph{extension} of inverse semigroups we mean a sequence
\begin{equation}\label{eq:extension}
K\xrightarrow{\,\,\iota\,\,} T\xrightarrow{\,\,\p\,\,} S
\end{equation}
 of idempotent bijective homomorphisms with $\iota(K)=\ker \p=\pinv(E(S))$.  We call the extension \emph{abelian} if $K$ is commutative.

Abelian extensions are classified by Lausch's second cohomology group~\cite[Section~7]{lausch}.  We recall the setup.  Let $S$ be an inverse semigroup.  An \emph{$S$-module} consists of a commutative inverse semigroup $K$, an idempotent bijective homomorphism $p\colon K\to E(S)$ and a (total) left action of $S$ on $K$ (by endomorphisms) such that $p(sk) = sp(k)s^*$ and $p(k)k=k$ for all $k\in K$ and $s\in S$.  The category of $S$-modules is an abelian category with enough projectives and injectives and Lausch developed a corresponding cohomology theory based on the derived functors of the functor taking an $S$-module $p\colon K\to E(S)$ to the $S$-equivariant sections $q\colon E(S)\to K$ of $p$ (with respect to the conjugation action on $E(S))$~\cite{lausch}.  We shall only need the second cohomology group, which classifies extensions.  Note that Lausch uses right modules, so we have dualized his results here.

Given an extension \eqref{eq:extension} of $S$ by a commutative inverse semigroup $K$ we can define an $S$-module structure on $K$ by putting $p=\p\iota\colon K\to E(S)$, choosing a set theoretic section $j\colon S\to T$ and setting \[sk = \iota^{-1}(j(s)\iota(k)j(s)^*)\] for $s\in S$ and $k\in K$.  Note that $\iota(K)=\ker \p$ is a normal inverse subsemigroup of $T$ and so this makes sense. Lausch shows that the module structure is independent of the choice of $j$.  Also each element of $t\in T$ can be uniquely written as $t=\iota(k)j(s)$ with $s\in S$ and $k\in K$, namely, $s=\p(t)$ and $k =\iota\inv (tj(\p(t))^*)$.

Given an $S$-module $K$ with $p\colon K\to E(S)$, a \emph{$2$-cocycle} is a mapping $c\colon S\times S\to K$ satisfying the following properties:
\begin{enumerate}
  \item $p(c(s,t)) = stt^*s^*$;
  \item $(sc(t,u))c(s,tu) = c(s,t)c(st,u)$ for $s,t,u\in S$.
\end{enumerate}
We say that $c$ is \emph{normalized} if $c(e,e)\in E(K)$ for all $e\in E(S)$.  The trivial $2$-cocycle is defined by $c(s,t)=(p|_{E(K)})\inv(stt^*s^*)$. The $2$-cocycles form an abelian group under pointwise multiplication with the trivial cocycle as the identity and with inversion done pointwise: $c\inv(s,t) = c(s,t)^*$.  The group of $2$-cocycles will be denoted $Z^2(S,K)$.  The set of all mappings $F\colon S\to K$ with $p(F(s))=ss^*$ is an abelian group $C^1(S,K)$ under pointwise operations with identity $F(s) = (p|_{E(K)})\inv (ss^*)$.  There is a coboundary homomorphism $\delta\colon C^1(S,K)\to Z^2(S,K)$ given by $(\delta F)(s,t) = F(s)(sF(t))F(st)^*$ and the image $B^2(S,K)$ is the group of \emph{$2$-coboundaries}.  We put $H^2(S,K)=Z^2(S,K)/B^2(S,K)$ and call it the second Lausch cohomology group of $S$ with coefficients in $K$\footnote{In Lausch's cohomology theory he calls this $H^2(S^I,K^0)$ where $S^I$ and $K^0$ are obtained by adjoining identities (see~\cite[Section~7]{lausch}), but we avoid this extra notation for simplicity.}.

\begin{Prop}\label{p:normalized.cocycle}
Every $2$-cocycle is cohomologous to a normalized one.
\end{Prop}
\begin{proof}
Let $c\colon S\times S\to K$ be a $2$-cocycle. Define $F\colon S\to K$ by $F(s) = c(ss^*,ss^*)^*$ and put $c'=c\delta F$.  Then, for $e\in E(S)$, we have $c'(e,e) = c(e,e)F(e)(eF(e))F(e^2)^* = c(e,e)c(e,e)^*c(e,e)^*c(e,e)\in E(K)$.  Thus $c'$ is normalized.
\end{proof}

Given an abelian extension as per \eqref{eq:extension}, we may choose a set theoretic section $j\colon S\to T$ such that $j|_{E(S)} = (\p|_{E(T)})\inv$.  We can then define $c\colon S\times S\to K$ by $c(s,s') = \iota\inv(j(s)j(s')j(ss')^*)\in K$; equivalently, $j(s)j(s')=\iota(c(s,s'))j(ss')$ for all $s,s'\in S$.  One checks, cf.~\cite[Section~7]{lausch}, that $c$ is a $2$-cocycle. Moreover, it is normalized since $c(e,e) = \iota\inv (j(e)j(e)j(ee)^\ast)=\iota\inv(j(e))\in E(K)$, as $j(e)\in E(T)$.  Changing the section $j$ results in a cohomologous $2$-cocycle.  Lausch does not require $j$ to preserve idempotents, which results in a not necessarily normalized $2$-cocycle, but we find normalized $2$-cocycles more convenient to work with.  So to each extension of $K$ by $S$, we obtain an $S$-module structure on $K$ and a normalized $2$-cocycle whose cohomology class is well defined.

Two extensions
\begin{gather*}
K\xrightarrow{\,\,\iota\,\,} T\xrightarrow{\,\,\p\,\,} S,\qquad
K\xrightarrow{\,\,\iota'\,\,} T'\xrightarrow{\,\,\p\,\,} S
\end{gather*}
are \emph{equivalent}  if there is an isomorphism $\psi\colon T\to T'$ such that the diagram
\[\begin{tikzcd}
K\arrow{r}{\iota}\arrow{d}{1_K} & T\arrow{r}{\p}\arrow{d}{\psi} & S\arrow{d}{1_S}\\
K\arrow{r}{\iota'} & T'\arrow{r}{\p'} & S
\end{tikzcd}\]
commutes.

Lausch proves that two extensions of $K$ by $S$ are equivalent if and only if the module structures on $K$ are the same and the corresponding cohomology classes of $2$-cocycles are the same.  Moreover, for each $2$-cocycle, he shows that there is an extension \eqref{eq:extension} of $K$ by $S$ realizing the cohomology class of the $2$-cocycle.  If $c\colon S\times S\to K$ is a normalized $2$-cocycle, we can put $T=\{(k,s)\in K\times S\mid p(k)=ss^*\}$ with multiplication given by $(k,s)(k',s') = (k (sk')c(s,s'),ss')$.  Here, $\iota(k) = (k,p(k))$ and $\p(k,s)=s$.  The inverse is given by  $(k,s)^*=((s^*k^*)c(s,s^*)^*,s^*)$.  The class of the trivial cocycle corresponds to the split extension of $K$ by $S$ and $T$, which, in this case, is what is termed the \emph{full restricted semidirect product} $K\bowtie S$ of $K$ by $S$ in~\cite[Chapter~5.3]{Lawson}.  More details can be found in~\cite[Section~7]{lausch}.

We record here some properties of normalized $2$-cocycles for later use.

\begin{Prop}\label{p:props.normalized.cocycle}
Let $K$ be an $S$-module with idempotent bijective homomorphism $p\colon K\to E(S)$ and $c\colon S\times S\to K$ a normalized $2$-cocycle.
\begin{enumerate}
\item $c(s,s^*s)=c(ss^*,s)\in E(K)$.
\item $c(s,s^*) = sc(s^*,s)$ for all $s\in S$.
\item $c(e,ef)= c(ef,e)\in E(K)$ for all $e,f\in E(S)$.
\item $c(e,f)\in E(K)$ for all $e,f\in E(S)$.
\item $c(s,e)=c(s,s^*se)$ for all $e\in E(S)$, $s\in S$.
\item $c(e,s) = c(ess^*,s)$ for all $e\in E(S)$, $s\in S$.
\item $c(e,es), c(se,e)\in E(K)$ for all $e\in E(S)$ and $s\in S$.
\item $c(e,s) = c(s,s^*es)$ for all $e\in E(S)$ and $s\in S$.
\item $c(s,e) = c(ses^*,s)$ for all $e\in E(S)$ and $s\in S$.
\item $c(u,s)c(ut,s^*u^*us)^* = (uc(t,s^*s)^*)c(u,t)$ if $s\leq t\in S$ and $u\in S$.
\item $c(s,u)c(tu,u^*s^*su)^* = c(t,s^*s)^*c(t,u)$ if $s\leq t\in S$ and $u\in S$.
\end{enumerate}
\end{Prop}
\begin{proof}
 For $e\in E(S)$, let $k_e$ denote the unique idempotent of $K$ with $p(k_e)=e$.
By the $2$-cocycle condition, the definition of a module and since $c$ is normalized, $c(s,s^*s)= (sc(s^*s,s^*s))c(s,s^*s)=c(s,s^*s)c(s,s^*s)$.  Also, the $2$-cocycle condition and the definition of a module yields  \[c(ss^*,s)^2=(ss^*c(ss^*,s))c(ss^*,s) =c(ss^*,ss^*)c(ss^*,s)=c(ss^*,s),\] as $c$ is normalized.  Thus $c(s,s^*s),c(ss^*,s)$ are idempotents with image $ss^*$ under $p$ and hence are equal as $p$ is idempotent bijective.  This yields (1).

For (2), we have $(sc(s^*,s))c(s,s^*s)=c(s,s^*)c(ss^*,s)$.  By the first item, we conclude that $c(s,s^*)=sc(s^*,s)$.

 For (3), we compute $c(e,ef)^2=(ec(e,ef))c(e,ef) = c(e,e)c(e,ef)=c(e,ef)$ since $c$ is normalized and \[c(e,ef)=efc(e,ef)=ec(e,ef)=c(e,e)c(e,ef).\]    Similarly, we have \[c(ef,e)=(efc(e,e))c(ef,e)=c(ef,e)c(ef,e)\] as $efc(e,e)=k_{ef}$ because $c$ is normalized.  Since $p(c(e,ef))=ef=p(c(ef,e))$ and $p$ is idempotent bijective, we deduce that $c(e,ef)=c(ef,e)$.

To prove (4), observe that $(ec(f,ef))c(e,ef)=c(e,f)c(ef,ef)=c(e,f)$ since $c$ is normalized, and so $c(e,f)\in E(K)$ by (3).

For (5), we have $(sc(s^*s,e))c(s,s^*se) = c(s,s^*s)c(s,e)$.  By (4), $sc(s^*s,e) = k_{ses^*}$ and, by (1), $c(s,s^*s) = k_{ss^*}\geq k_{ses^*}$, and so the left hand side is $c(s,s^*se)$ and the right hand side is $c(s,e)$.

For (6), we have $(ec(ss^*,s))c(e,s)=c(e,ss^*)c(ess^*,s)$.  As $c(ss^*,s)=k_{ss^*}$, $c(e,ss^*)=k_{ess^*}$ by (1) and (4), we deduce that $c(e,s)=c(ess^*,s)$.

For (7), we have $c(e,es) = c(ess^*,es)\in E(K)$ and $c(se,e)=c(se,es^*s)\in E(K)$ by (6) and (1).

For (8), we note by the $2$-cocycle condition \[(ec(s,s^*es))c(e,es)=c(e,s)c(es,s^*es).\]  But by (1), $c(es,s^*es)=k_{ess^*}$  and by (7) $c(e,es)=k_{ess^*}$.  Also since $p(c(s,s^*es)) = ess^*\leq e$, we have that $ec(s,s^*es)=ess^*c(s,s^*es)=c(s,s^*es)$.  We conclude that $c(s,s^*es)=c(e,s)$ as required.

For (9), we note that $(ses^*c(s,e))c(ses^*,se) = c(ses^*,s)c(se,e)$ by the $2$-cocycle condition.  But $c(ses^*,se)=k_{ses^*}=c(se,e)$  by (1) and (7), and so we deduce, since $p(c(s,e)) = ses^*$, that $c(s,e)=c(ses^*,s)$.

To prove (10), we compute
\begin{align*}
(uc(t,s^*s))c(u,s) &= (uc(t,s^*s))c(u,ts^*s) = c(u,t)c(ut,s^*s)\\ &= c(u,t)c(ut, t^*u^*uts^*s)
\end{align*}
 where the last equality uses (5).  But $t^*u^*uts^*s=s^*st^*u^*uts^*s=s^*u^*us$.  Thus we have $(uc(t,s^*s))c(u,s)=c(u,t)c(ut,s^*u^*us)$.  Using $p(uc(t,s^*s)) = uss^*u^*=p(c(ut,s^*u^*us))$, we deduce that \[c(u,s)c(ut,s^*u^*us)^* = (uc(t,s^*s)^*)c(u,t),\] as required.

We now turn to (11).  From the $2$-cocycle condition we obtain
\begin{align*}
(tc(u,u^*s^*su))c(t,s^*su) &= c(t,u)c(tu,u^*s^*su), \\
(tc(s^*s,u))c(t,s^*su) &= c(t,s^*s)c(ts^*s,u) \\ &= c(t,s^*s)c(s,u)
\end{align*}
Note that $c(s^*s,u) = c(u,u^*s^*su)$ by (8).  Therefore, $c(t,s^*s)c(s,u)= c(t,u)c(tu,u^*s^*su)$.  Then, using that we have
\begin{align*}
p(c(t,s^*s))&=ss^*\geq suu^*s^*=p(c(tu,u^*s^*su))\\ &=p(c(s,u)) = p(c(t,s^*s)^*c(t,u)),
\end{align*}
 we deduce that $c(s,u)c(tu,u^*s^*su)^* = c(t,s^*s)^*c(t,u)$, as required.
\end{proof}

\subsection{Twists and extensions of groupoids}
By an \emph{extension} of ample groupoids we mean an exact sequence
\begin{equation}\label{eq:ext.groupoid}
\mathscr K\xrightarrow{\,\,\iota\,\,} \mathscr H\xrightarrow{\,\,\p\,\,} \mathscr G
\end{equation}
 of open iso-unital functors with $\iota$ injective, $\p$ surjective and $\p^{-1}(\mathscr G\skel 0) = \iota(\mathscr K)$.  A second extension \[\mathscr K\xrightarrow{\,\,\iota'\,\,} \mathscr H'\xrightarrow{\,\,\p'\,\,} \mathscr G\] is \emph{equivalent} to \eqref{eq:ext.groupoid} if there is an isomorphism $\psi\colon \mathscr H\to \mathscr H'$ making the diagram
\[\begin{tikzcd}
\mathscr K\arrow{r}{\iota}\arrow{d}{1_{\mathscr K}} & \mathscr H\arrow{r}{\p}\arrow{d}{\psi} & \mathscr G\arrow{d}{1_{\mathscr G}}\\
\mathscr K\arrow{r}{\iota'} & \mathscr H'\arrow{r}{\p'} & \mathscr G
\end{tikzcd}\]
commute.

We now wish to show that our functors $\Gamma_c$ and $\mathcal G$ are exact in the sense that they preserve extensions.

\begin{Thm}\label{t:preserve.extensions}
The functors $\Gamma_c$ and $\mathcal G$ preserve extensions. More precisely:
\begin{enumerate}
\item If $\mathscr K\xrightarrow{\,\,\iota\,\,} \mathscr H\xrightarrow{\,\,\p\,\,} \mathscr G$ is an extension of ample groupoids, then \[\Gamma_c(\mathscr K)\xrightarrow{\,\,\Gamma_c(\iota)\,\,} \Gamma_c(\mathscr H)\xrightarrow{\,\,\Gamma_c(\p)\,\,} \Gamma_c(\mathscr G)\] is an extension of Boolean inverse semigroups.
\item If $K\xrightarrow{\,\,\iota\,\,} T\xrightarrow{\,\,\p\,\,} S$ is an extension of Boolean inverse semigroups, then \[\mathcal G(K)\xrightarrow{\,\,\mathcal G(\iota)\,\,} \mathcal G(T)\xrightarrow{\,\,\mathcal G(\p)\,\,} \mathcal G(S)\] is an extension of ample groupoids.
\end{enumerate}
\end{Thm}
\begin{proof}
By Theorem~\ref{t:equiv.cats} and Proposition~\ref{p:preserve.epimono}, all that remains to check is that $\Gamma_c(\iota)(\Gamma_c(\mathscr K))=\Gamma_c(\p)\inv(E(\Gamma_c(\mathscr G)))$ and $\mathcal G(\iota)(\mathcal G(K)) = \mathcal G(\p)\inv (\mathcal G(S)\skel 0)$.  For the first item, note that since $\mathscr G\skel 0$ is open, we have that $\p\inv(\mathscr G\skel 0)$ is an open subgroupoid of $\mathscr H$ and hence ample. Moreover, $\Gamma_c(\p)\inv(E(\Gamma_c(\mathscr G)))=\Gamma_c(\pinv(\mathscr G\skel 0))$ .  Since $\iota\colon \mathscr K\to \p\inv(\mathscr G\skel 0)$ is surjective and open iso-unital, we deduce from Proposition~\ref{p:preserve.epimono} that \[\Gamma(\iota)(\Gamma_c(\mathscr K)) = \Gamma_c(\p\inv(\mathscr G\skel 0))=\Gamma_c(\p)\inv(E(\Gamma_c(\mathscr G))),\] as required.

For the second item, liet $[k,\chi]\in \mathcal G(K)$.  Then \[\mathcal G(\p)\mathcal G(\iota)([k,\chi]) = [\p\iota(k),\chi\circ (\p\iota|_{E(K)})\inv]\in \mathcal G(S)\skel 0\] since $\p\iota(k)\in E(S)$.  Conversely, if $[\p(t),\chi\circ (\p|_{E(T)})\inv]=\mathcal G(\p)([t,\chi])\in \mathcal G(S)\skel 0$, then there is an idempotent $f\in E(S)$ with $f\leq \p(t)$ and \[\chi\circ (\p|_{E(S)})\inv(f)=1.\]  Putting $e=(\p|_{E(S)})\inv (f)$, we have that $\chi(e)=1$ and $\p(te)=\p(t)f=f$.  Since $\chi(et^*te)=1$, we deduce that $[t,\chi]=[te,\chi]$.  As $te\in \pinv(E(S))$, we have that $te=\iota(k)$ with $k\in K$.  Then $[k,\chi\circ \iota|_{E(K)}]\in \mathcal G(K)$ (as $\chi\iota(k^*k) = \chi(et^*te)=1$) and $\mathcal G(K)(\iota)([k,\chi\circ\iota|_{E(K)}]) = [te,\chi]=[t,\chi]$.  Thus $\mathcal G(\iota)(\mathcal G(K))= \mathcal G(\p)\inv(\mathcal G(S)\skel 0)$, as required.  This completes the proof.
\end{proof}

It follows from Theorem~\ref{t:preserve.extensions} that classifying extensions of ample groupoids is equivalent to classifying extensions of Boolean inverse semigroups.  One has to have some care though in applying Lausch cohomology in order to make sure that the extension coming from a $2$-cocycle is Boolean.

We briefly digress to examine how the Hausdorff property behaves under extensions of ample groupoids.  It follows from the results of~\cite{groupoidprimitive}, see also~\cite{LawsonLenz}, that an ample groupoid $\mathscr G$ is Hausdorff if and only if each $U\in \Gamma_c(\mathscr G)$ has a maximum idempotent below it in the natural partial order.  In the case $\mathscr G$ is Hausdorff, $U\cap\mathscr G\skel 0$ is the maximum idempotent below $U$.

\begin{Prop}\label{p:ext.hausdorff}
Consider an extension of ample groupoids as in \eqref{eq:ext.groupoid}.  Then $\mathscr G$ is Hausdorff if and only if each element of $\Gamma_c(\mathscr H)$ has a maximum element of $\Gamma_c(\iota)(\Gamma_c(\mathscr K))$ below it.  Moreover, if $\mathscr G$ and $\mathscr K$ are both Hausdorff, then so is $\mathscr H$.
\end{Prop}
\begin{proof}
Without loss of generality, we may assume that $\iota$ is an inclusion.  Suppose that $\mathscr G$ is Hausdorff and $U\in \Gamma_c(\mathscr H)$.  Then $\p(U)\in \Gamma_c(\mathscr G)$ has a unique maximum idempotent $V\subseteq \p(U)$.  Since $V\subseteq \mathscr G\skel 0$, we have that $\pinv(V)\subseteq \mathscr K$.  Also, since $V$ is clopen (as $\mathscr G$ is Hausdorff), we have that $\pinv(V)\cap U\subseteq\mathscr K$ is clopen and hence compact (since $U$ is compact).  Thus $W=\pinv(V)\cap U\in \Gamma_c(\mathscr K)$ and $W\subseteq U$. Moreover, if $W'\in \Gamma_c(\mathscr K)$ with $W'\subseteq U$, then $\p(W')\subseteq \p(U)\cap \mathscr G\skel 0=V$ and so $W'\subseteq W$.  Thus $W$ is the maximum element of $\Gamma_c(\mathscr K)$ below $U$.  For the converse, by Proposition~\ref{p:preserve.epimono}, every element of $\Gamma_c(\mathscr G)$ is of the form $\p(U)$ with $U\in \Gamma_c(\mathscr H)$.  Let $W$ be the maximum element of $\Gamma_c(\mathscr K)$ below $U$. Then $\p(W)\subseteq\mathscr G\skel 0$ is an idempotent and $\p(W)\subseteq \p(U)$.  Suppose that $V\subseteq\mathscr G\skel 0$ is compact open with $V\subseteq \p(U)$.  Let $x\in V$ and let $h\in U$ with $\p(h)=x$.  Since $U\cap \pinv(V)$ is an open subset of $\mathscr K$ containing $h$, we can find a compact open bisection $W'\in \Gamma_c(\mathscr K)$ with $h\in W'\subseteq U\cap \pinv(V)\subseteq \mathscr K$.  Then $W'\subseteq W$ by choice of $W$ and so $\p(W')\subseteq \p(W)$.  As $x=\p(h)\in \p(W')$, we deduce that $x\in \p(W)$ and hence $V\subseteq \p(W)$.  This shows that $\p(W)$ is the maximum idempotent below $\p(U)$ and hence $\mathscr G$ is Hausdorff.

Suppose now that $\mathscr G$ and $\mathscr K$ are Hausdorff.  We continue to assume that $\iota$ is an inclusion.  If $U\in\Gamma_c(\mathscr H)$, then there is a unique maximum element $V$ of $\Gamma_c(\mathscr K)$ with $V\subseteq U$ by the above paragraph. Since $\mathscr K$ is Hausdorff, there is a unique maximum idempotent $W\subseteq V$ in $\Gamma_c(\mathscr K)$.  Note that $E(\Gamma_c(\mathscr K)) = E(\Gamma_c(\mathscr H))$ as $\mathscr K\skel 0=\mathscr H\skel 0$.  Thus any idempotent $U'\subseteq U$ belongs to $\Gamma_c(\mathscr K)$, and hence is below $V$, and therefore below $W$.  Thus $W$ is the maximum idempotent below $U$.  It follows that $\mathscr H$ is Hausdorff.
\end{proof}

Let $\mathscr G$ be an ample groupoid and $A$ a discrete abelian group with identity $1$ and with binary operation written multiplicatively (which we can think of as a one-object ample groupoid).  In practice, $A=R^\times$ where $R$ is a commutative ring with unity.  A \emph{(discrete) $A$-twist} over $\mathscr G$ is an exact sequence of ample groupoids
\begin{equation}\label{eq:Atwist}
A\times \mathscr G\skel 0\xrightarrow{\,\,\iota\,\,} \Sigma\xrightarrow{\,\,\p\,\,} \mathscr G
\end{equation}
(with $\p\iota(a,x)=x$) which is \emph{central} in the sense that
\begin{equation}\label{eq:central}
\iota(a,\ran(\p(s)))s=s\iota(a,\dom(\p(s)))
\end{equation}
for all $a\in A$ and $s\in \Sigma$.  In this case, we often write $as$ for $\iota(a,\ran(\p(s)))s=s\iota(a,\dom(\p(s)))$ and we note that $(as)(a's')  = (aa')ss'$ for $a,a'\in A$ and $s,s'\in \Sigma$ with $\dom(s)=\ran(s')$.  Also note that $(a,s)\mapsto as$ is a free action of $A$ on $\Sigma$ by homeomorphisms (but not by groupoid automorphisms).  Note that since $A\times \mathscr G\skel 0$ is Hausdorff, if $\mathscr G$ is Hausdorff, then so is $\Sigma$ by Proposition~\ref{p:ext.hausdorff}.  This was already observed in~\cite[Corollary~2.3]{twistedreconstruction} using an argument specific to twists.

Note that a different definition of twists is given in~\cite{twists}, where $\Sigma$ is not required to be ample and $\iota$ is not required to be open, but where $\p$ is required to be a locally trivial fiber bundle with fiber $A$.  However, it is shown in~\cite[Proposition~2.2]{twistedreconstruction} that our definition is equivalent to the one in~\cite{twists}.  The first step in reformulating $A$-twists in terms of inverse semigroups is to identify $\Gamma_c(A\times \mathscr G\skel 0)$ in semigroup theoretic terms.  Note that $A\cup \{0\}$ is an inverse semigroup, where $0a=0=a0=0^2$ for all $a\in A$.  Note $a^*=a^{-1}$ if $a\in A$ and $0^*=0$.

\begin{Prop}\label{p:bisections}
Let $\til A =C_c(\mathscr G\skel 0,A\cup \{0\})$ be the set of locally constant mappings $f\colon \mathscr G\skel 0\to A\cup \{0\}$ with compact support (meaning $\supp(f)=f^{-1}(A)$ is compact).  Then $\til A$ is a commutative inverse semigroup under pointwise multiplication and there is an isomorphism $\gamma\colon \til A\to \Gamma_c(A\times \mathscr G\skel 0)$ given by $f\mapsto U_f=\{(f(x),x)\mid x\in \supp(f)\}$.
\end{Prop}
\begin{proof}
Clearly $\til A$ is a commutative inverse semigroup with $f^*(x) = f(x)^\ast$ giving the inversion.  Note that \[\supp(fg) = \supp(f)\cap \supp(g)=\supp(f)\supp(g)\] since $A$ is a group.  The map $\gamma$ is a homomorphism because
\begin{align*}
U_fU_g &= \{(f(x),x)(g(x),x)\mid x\in \supp(f)\cap \supp(g)\} \\ &= \{(f(x)g(x),x)\mid  x\in \supp(fg)\}=U_{fg}.
\end{align*}
  It is obvious that $\gamma$ is injective since $\supp(f)=\dom(U_f)$ and $f = \pi\circ (\dom|_{U_f})^{-1}$ on its support, where $\pi$ is the projection to $A$.  Conversely, given $U\in \Gamma_c(A\times \mathscr G\skel 0)$, we can define $f$ by
\[f(x) = \begin{cases}\pi\circ (\dom|_U)\inv (x), & \text{if}\ x\in \dom(U)\\ 0, & \text{else}\end{cases}\] and $U=U_f$.  Thus $\gamma$ is an isomorphism.
\end{proof}

We can make $\til A$ into a $\Gamma_c(\mathscr G)$-module as follows.  Define $p\colon \til A\to E(\Gamma_c(\mathscr G))$ by $p(f)= \supp(f)$.  The action is given by
\[(Uf)(x) = \begin{cases} f(\dom((\ran|_U)\inv(x))), & \text{if}\ x\in \ran(U)\\ 0, & \text{else}\end{cases}\] for $U\in \Gamma_c(\mathscr G)$ and $x\in \mathscr G\skel 0$.

\begin{Prop}\label{p:ismodule}
The commutative inverse semigroup $\til A$ is a $\Gamma_c(\mathscr G)$-module with respect to $p$ and the action $(U,f)\mapsto Uf$.
\end{Prop}
\begin{proof}
The idempotents of $\til A$ are the characteristic functions $1_U$ with $U\subseteq \mathscr G\skel 0$ compact open and $p(1_U)=U$.  Since $\supp(fg)=\supp(f)\cap \supp(g)=\supp(f)\supp(g)$, we have that $p$ is an idempotent bijective semigroup homomorphism.  If $W\subseteq \mathscr G\skel 0$ is compact open and $U\in \Gamma_c(\mathscr G)$, then $UWU^*$ consists of those $x\in \mathscr G\skel 0$ such that there is an arrow $g\colon y\to x$ in $U$ with $y\in W$.  The support of $Uf$ consists of those $x\in \mathscr G\skel 0$ such that there is an arrow $g\colon y\to x$ with $f(y)\neq 0$.  Hence $\supp(Uf) = U\supp(f)U^*$.  Also notice that if $U\subseteq \mathscr G\skel 0$ is compact open, then $Uf = 1_Uf$.  It follows that $p(f)f= 1_{\supp(f)}f=f$.  It remains to check that $(U,f)\mapsto Uf$ is a semigroup action by endomorphisms.

Clearly, $U(fg)=(Uf)(Ug)$ from the definition.  If $U,V\in \Gamma_c(\mathscr G)$, then $x\in \supp(U(Vf))$ if and only if there is $g\colon y\to x$ with $g\in U$ and $h\in V$ with $h\colon z\to y$ with $f(z)\neq 0$, in which case $(U(Vf))(x)=f(z)$.  But $(UV)f(x)$ is non-zero if and only if there is $k\colon z\to x$ in $UV$ with $f(z)\neq 0$, in which case $(UV)f(x) = f(z)$.  But then $k=gh$ for unique $g\in U$ and $h\in V$ with, say, $g\colon y\to x$ and $h\colon z\to y$.  We deduce that $(UV)f(x) = U(Vf(x))$.  This concludes the proof that $\til A$ is a $\Gamma_c(\mathscr G)$-module.
\end{proof}

Given an extension \eqref{eq:Atwist}, the question of whether it is central (i.e., satisfies \eqref{eq:central}) can be determined from the corresponding $\Gamma_c(\mathscr G)$-module structure on $\Gamma_c(A\times \mathscr G\skel 0)$ in the inverse semigroup extension
\begin{equation}\label{eq:twist.inverse.version}
\Gamma_c(A\times\mathscr G\skel 0)\xrightarrow{\Gamma_c(\iota)} \Gamma_c(\Sigma)\xrightarrow{\Gamma_c(\p)} \Gamma_c(\mathscr G).
\end{equation}

\begin{Prop}\label{p:is.central}
An extension \eqref{eq:Atwist} is central if and only if the semigroup isomorphism $\gamma\colon \til A\to \Gamma_c(A\times \mathscr G\skel 0)$ of Proposition~\ref{p:bisections} is an isomorphism of $\Gamma_c(\mathscr G)$-modules with respect to the module structure on $\Gamma_c(A\times \mathscr G\skel 0)$ coming from the extension \eqref{eq:twist.inverse.version}.
\end{Prop}
\begin{proof}
Note that $\p\iota(\gamma(f))=\p(\iota(U_f)) = \supp(f)$ by construction.  So we just need to check that $\Gamma_c(\mathscr G)$-equivariance is equivalent to centrality.  Fix a section $j\colon \Gamma_c(\mathscr G)\to \Gamma_c(\Sigma)$ of $\Gamma_c(\p)$ with $j|_{E(\Gamma_c(\mathscr G))}= (\Gamma_c(\p)|_{E(\Gamma_c(\Sigma))})\inv$.  Then the action of $V\in \Gamma_c(\mathscr G)$ on $\gamma(f)=U_f$ is given by \[VU_f = \iota\inv(j(V)\iota(U_f)j(V)^*).\]  This set consists of all $(a,x)$ such that there is $y\in \supp(f)$ and $s\in j(V)$ with $\p(s)\colon y\to x$ and $(a,x)=\iota\inv(s\iota(f(y),y)s\inv)$.  On the other hand, $(Vf)(x)\neq 0$ if and only if there is $g\in V$ and $y\in \supp(f)$, with $g\colon y\to x$ and then $(Vf)(x)=f(y)$.  Since $\p|_{j(V)}\colon j(V)\to V$ is a homeomorphism by Proposition~\ref{p:local.homeo}, this is equivalent to there being $s\in j(V)$ and $y\in \supp(f)$ with $\p(s)\colon y\to x$ and then $(Vf)(x)=f(y)$.  Thus $VU_f=U_{Vf}=\gamma(Vf)$ if and only if $\iota(f(\dom(\p(s))),\ran(\p(s)))=s\iota(f(\dom(\p(s))),\dom(\p(s)))s\inv$ for all $s\in j(V)$ with $\dom(\p(s))\in \supp(f)$.  This equivalent to
\begin{equation}\label{eq:check.equiv}
\iota(f(\dom(\p(s))),\ran(\p(s)))s= s\iota(f(\dom(\p(s))),\dom(\p(s)))
\end{equation}
for all $s\in j(V)$ with $\dom(\p(s))\in \supp(f)$, which clearly is satisfied if the extension is central.  Hence if the extension is central, $\gamma$ is an isomorphism of modules.

Conversely, if $\gamma$ is a module isomorphism and $s\in \Sigma$, $a\in A$, then we can choose a bisection $V$ containing $\p(s)$ and we may assume that $j(V)$ was chosen to contain $s$.  We can define $f = a1_{\dom(V)}$. Then, by \eqref{eq:check.equiv}, the equivariance of $\gamma$ implies that $\iota(a,\ran(\p(s)))s= \iota(f(\dom(\p(s))),\ran(\p(s)))s =s\iota(f(\dom (\p(s))),\dom(\p(s))) = s\iota(a,\dom(\p(s)))$.  We conclude that the extension is central.
\end{proof}

We may now deduce from Theorem~\ref{t:equiv.cats}, Theorem~\ref{t:preserve.extensions} and Proposition~\ref{p:is.central} the following.

\begin{Cor}\label{c:classify.part}
There is a bijection between equivalence classes of $A$-twists over $\Gamma$ and extensions $\til A\to T\to\Gamma_c(\mathscr G)$ with $T$ a Boolean inverse semigroup and with the $\Gamma_c(\mathscr G)$-module structure on $\til A$ as per Proposition~\ref{p:ismodule}.
\end{Cor}

To complete our classification of $A$-twists by cohomology classes, we need that every extension of $\til A$ by $\Gamma_c(\mathscr G)$ is a Boolean inverse semigroup. In fact, it turns out every extension of a commutative Boolean inverse semigroup by a Boolean inverse semigroup is Boolean.

%
%

\begin{Prop}\label{p:all.good}
Let $S$ be a Boolean inverse semigroup and $K$ a commutative Boolean inverse semigroup.  If $K\xrightarrow{\,\,\iota\,\,} T\xrightarrow{\,\,\p\,\,} S$ is any extension of $S$ by $K$, then $T$ is a Boolean inverse semigroup.
\end{Prop}
\begin{proof}
Choose a section $j\colon S\to T$ with $j|_{E(S)} = (\p|_{E(T)})\inv$.  We can then turn $K$ into an $S$-module (with $p=\p\iota\colon K\to E(S)$);   denote by $c\colon S\times S\to K$ the corresponding normalized $2$-cocycle.  Since $E(T)\cong E(S)$ is a Boolean algebra, we just need to show that if $t,u\in T$ are orthogonal, then they have a join.  In fact, it is enough to show they have a common upper bound $y$ for then, by~\cite[Proposition~1.4.18]{Lawson}, one has $y(t^*t\vee u^*u) = yt^*t\vee yu^*u = t\vee u$. Put $\til t=\p(t)$ and $\til u=\p(u)$.  Then the fact that $t,u$ are orthogonal is equivalent to $t^*t,u^*u$ and $tt^*,uu^*$ being orthogonal.  This in turn is equivalent to $\til t,\til u$ being orthogonal since $\p$ is idempotent bijective.  Thus $x=\til t\vee \til u$ exists.  We can write $t=\iota(k_t)j(\til t)$ and $u=\iota(k_u)j(\til u)$ for unique $k_t,k_u\in K$ with $p(k_t) = \til t\til t^*$ and $p(k_u)=\til u\til u^*$.  Then $p(k_tc(x,\til t^*\til t)^*) = \til t\til t ^*$ and $p(k_uc(x,\til u^*\til u)^*)=\til u\til u^*$ and hence, since $p$ is idempotent bijective, we have $k=k_tc(x,\til t^*\til t)^* \vee k_uc(x,\til u^*\til u)^*$ is defined and $p(k)=\til t\til t^*\vee \til u\til u^*=  xx^*$ (since $p$ is additive, being idempotent bijective).  Let $y=\iota(k)j(x)$.  We claim that $t,u\leq y$.  The argument is symmetric, so we just handle the case of $t$.

Note that $\iota(K)$ centralizes $E(T)$ as $\iota(k)e=\iota (kf)=\iota(fk)=e\iota(k)$ where $f\in E(K)$ is unique with $\iota(f)=e$ (using that $\iota$ is idempotent bijective).  We then have \[yt^*t = \iota(k)j(x)j(\til t^*\til t)=\iota(k)\iota(c(x,\til t^*\til t))j(x\til t^*\til t) = \iota(kc(x,\til t^*\til t))j(\til t).\]  Therefore, it suffices to show that $kc(x,\til t^*\til t) = k_t$.  Note that $p(c(x,\til t^*\til t))= \til t^*\til t$ and so $k_uc(x,\til u^*\til u)^*c(x,\til t^*\til t) =0$. Therefore, $kc(x,\til t^*\til t) = k_tc(x, \til t^*\til t)^*c(x, \til t^*\til t) =k_t$, since multiplication distributes over joins in a Boolean inverse semigroup, as required.  This completes the proof.
\end{proof}

Putting together Theorem~\ref{t:equiv.cats}, Theorem~\ref{t:preserve.extensions}, Corollary~\ref{c:classify.part}, Proposition~\ref{p:all.good},  and~\cite[Section~7]{lausch}, we have the following classification of $A$-twists.

\begin{Thm}\label{t:classify}
Let $\mathscr G$ be an ample groupoid and $A$ a discrete abelian group.  Then there is a bijection between equivalence classes of discrete $A$-twists and $H^2(\Gamma_c(\mathscr G), C_c(\mathscr G\skel 0, A\cup \{0\}))$.
\end{Thm}

Note that the class of the trivial $2$-cocycle corresponds to the trivial extension $A\times \mathscr G\skel 0\to A\times \mathscr G\to \mathscr G$.  Indeed, $\Gamma_c(A\times \mathscr G)$ can be identified with the full restricted semidirect product $\til A\bowtie \Gamma_c(\mathscr G)$ via the mapping that sends $(f,U)\in \til A\bowtie \Gamma_c(\mathscr G)$ with $\supp(f)=\ran(U)$ to the compact open bisection $\{(f(\ran(g)),g)\mid g\in U\}$ of $A\times \mathscr G$, as is easily checked.     The abelian group structure on $H^2(\Gamma_c(\mathscr G),C_c(\mathscr G\skel 0, A\cup \{0\}))$ corresponds to the Baer sum operation on $A$-twists over $\mathscr G$.
 We state the result, but omit many of the routine verifications as it will not be needed in the sequel.

If
\[\begin{tikzcd}
A\times \mathscr G\skel 0\arrow{r}{\iota} & \Sigma\arrow{r}{\p} & \mathscr G,  &
A\times \mathscr G\skel 0 \arrow{r}{\iota'}& \Sigma'\arrow{r}{\p'} & \mathscr G
\end{tikzcd}\]
are  $A$-twists,  then since $\p$ and $\p'$ are local homeomorphisms by Proposition~\ref{p:local.homeo}, the projections from the pullback  $\Sigma\times_{\p,\p'}\Sigma'$ along $\p,\p'$ to $\Sigma$ and $\Sigma'$ are local homeomorphisms and hence the pullback is an ample groupoid.
The \emph{Baer sum} $\Sigma\oplus \Sigma'$ is the groupoid $\Sigma\times _{\p,\p'}\Sigma'/A$ where $A$ acts on $\Sigma\times_{\p,\p'}\Sigma'$ by $a(s,s') = (a\inv s, as')$ and we use the quotient topology. This a properly discontinuous action of a discrete group and so the quotient map $(\Sigma\times _{\p,\p'}\Sigma')\to \Sigma\oplus \Sigma'$ is a covering map, hence a local homeomorphism.
This equivalence relation is compatible with multiplication and so gives a groupoid structure which makes $\Sigma\oplus \Sigma'$ an ample groupoid. We leave the straightforward verifications to the interested reader.   The inclusion $\kappa$ of $A\times \mathscr G\skel 0$ into $\Sigma\oplus \Sigma'$ sends $(a,x)$ to $[x,ax]$ and the quotient map $\rho$ from $\Sigma\oplus \Sigma'$ to $\mathscr G$ sends $[s,s']$ to $\p(s)=\p'(s')$. 

\begin{Thm}
The set of equivalence classes of $A$-twists over $\mathscr G$ under Baer sum is an abelian group isomorphic to $H^2(\Gamma_c(\mathscr G),C_c(\mathscr G\skel 0, A\cup \{0\}))$.
\end{Thm}
\begin{proof}
Using Theorem~\ref{t:classify} we just need to show that the cohomology class corresponding to the equivalence class of $\Sigma\oplus \Sigma'$ is the product of the cohomology classes corresponding to $\Sigma$ and $\Sigma'$.  Let $j\colon \Gamma_c(\mathscr G)\to \Gamma_c(\Sigma)$ and $j'\colon \Gamma_c(\mathscr G)\to \Gamma_c(\Sigma')$ be set theoretic sections with $j$ and $j'$ preserving idempotents.  Then the cohomology classes corresponding to $\Sigma$ and $\Sigma'$ are given by normalized $2$-cycles $c,c'$, respectively, with $j(U)j(V) = \iota(\gamma(c(U,V)))j(UV)$ and $j'(U)j'(V)=\iota'(\gamma(c'(U,V)))j'(UV)$ where $\gamma\colon \til A\to \Gamma_c(A\times \mathscr G\skel 0)$ is the isomorphism of Proposition~\ref{p:is.central}.  Define $\til j\colon \Gamma_c(\mathscr G)\to \Gamma_c(\Sigma\oplus \Sigma')$ by \[\til j(U) = \{[g,g']\mid g\in j(U), g'\in j'(U)\}.\]  This is a compact open set since it is the image of $j(U)\times_{\p,\p'} j'(U)$, which is compact open in $\Sigma\times_{\p,\p'}\Sigma'$.  It is a bisection because if $\dom([g,g']) = \dom([h,h'])$ with $g,h\in j(U)$ and $g',h'\in j'(U)$, then $\dom (g)=\dom(h)$ and $\dom(g')=\dom(h')$ and so $g=g'$ and $h=h'$.  Similarly, $\ran|_{\til j(U)}$ is injective.  By construction $\rho(\til j(U))=U$.  Also, if $U\subseteq \mathscr G\skel 0$, then $\til j(U)\subseteq (\Sigma\oplus \Sigma')\skel 0$.  Let us compute the normalized $2$-cocycle $\til c$ corresponding to $\til j$.

Let $U,V\in \Gamma_c(\mathscr G)$.  Then $\til j(U)\til j(V)(\til j(UV))^*$ consists of all composable products $[g,g'][h,h'][k\inv, (k')\inv]$ with $g\in j(U)$, $h\in j(V)$, $k\in j(UV)$, $g'\in j'(U)$, $h'\in j'(V)$ and $k'\in j'(UV)$.  Then since $U,V$ are bisections, it follows from composability that $\p(g)\p(h)=\p(k)$ and $\p'(g')\p'(h')=\p'(k')$.  Therefore, $gh=ak$ and $g'h'=a'k'$ for some $a,a'\in A$.  But using that $j(U)j(V) = \iota(\gamma(c(U,V)))j(UV)$ and $j'(U)j'(V)=\iota'(\gamma(c'(U,V)))j'(UV)$, we deduce that $a=c(U,V)(\ran(gh))$ and $a'=c'(U,V)(\ran(gh))$.  Thus
\begin{align*}
[g,g'][h,h'][k\inv, (k')\inv] &= [c(U,V)(\ran(gh))\ran(gh),c'(U,V)(\ran(gh)\ran(gh)]\\ &= [\ran(gh),(c(U,V)c'(U,V))(\ran(gh))\ran(gh)]\\ &= \kappa((c(U,V)c'(U,V))(\ran(gh)),\ran(gh)).
\end{align*}
 It follows that $\til c(U,V) = c(U,V)c'(U,V)$, as required.
\end{proof}

We remark that the easiest way to build an $A$-twist over $\mathscr G$ is to begin with a normalized locally constant $2$-cocycle $c\colon \mathscr G\skel 2\to A$ (where $\mathscr G\skel 2$ is the space of composable pairs of elements of $\mathscr G$).  Being a $2$-cocycle means that $c(h,k)c(g,hk)=c(g,h)c(gh,k)$ and being normalized means $c(x,x)=1$ for all units $x\in \mathscr G\skel 0$.  From a normalized $2$-cocycle, you can build a twist $\Sigma = A\times \mathscr G$ with the product topology, $\dom(a,g)=(1,\dom(g))$, $\ran(a,g)=(1,\ran(g))$ and $(a,g)(b,h) = (abc(g,h),gh)$.  The inverse is given by $(a,g)\inv = (a\inv c(g,g\inv)\inv, g\inv)$.  The corresponding exact sequence is
\[A\times \mathscr G\skel 0\to \Sigma\to \mathscr G\] where the first map is the inclusion and the second is the projection.  One easily adapts~\cite[Proposition~4.8]{twists} to general $A$ to show that an  $A$-twist as in \eqref{eq:Atwist} is equivalent to one coming from a locally constant $2$-cocycle on $\mathscr G$ if and only if there is a continuous section $s\colon \mathscr G\to \Sigma$ with $s(\mathscr G\skel 0)\subseteq \Sigma\skel 0$.  By Proposition~\ref{p:section} this occurs if and only if $\Gamma_c(\p)\colon \Gamma_c(\Sigma)\to \Gamma_c(\mathscr G)$ admits an order-preserving and idempotent-preserving section.  Proposition~\ref{p:section.sigma.compact} recovers the folklore result that any twist over a second countable Hausdorff ample groupoid comes from a $2$-cocycle~\cite{twists}, and extends it to the paracompact case. 

\section{Inverse semigroup crossed products}
It was observed in~\cite{GonRoy2017a} that Steinberg algebras of Hausdorff ample groupoids are skew inverse semigroup rings (see~\cite{Demeneghi} for the non-Hausdorff case).  We will now define a notion of inverses semigroup crossed product that captures twisted Steinberg algebras.  Even without a twist, our definition will look different than the definition of a skew inverse semigroup ring found in the literature~\cite{GonRoy2017a,simpleskew}, but it coincides with that definition for the case of so-called spectral actions~\cite{Skewasconv}, which is what arises in the ample groupoid setting.

\subsection{Actions}

Let $R$ be a ring (associative,  but not necessarily unital or commutative).  We say that a ring endomorphism $\psi\colon R\to R$ is \emph{proper} if $\psi(R)=Re$ with $e$ a central idempotent of $R$ (necessarily unique).  Note that if $e,f\in E(Z(R))$ are central idempotents, then $Re\cap Rf=Ref=ReRf$.

\begin{Prop}\label{p:proper.endos}
The proper ring endomorphisms of $R$ form a semigroup $\End_c(R)$ under composition.  The idempotent proper endomorphisms are those of the form  $\p_e(r)=re$ with $e\in E(Z(R))$ and hence $E(\End_c(R))$ is a commutative subsemigroup isomorphic to the Boolean algebra $E(Z(R))$ of central idempotents of $R$ (under meet).  Moreover, $\p(Z(R))\subseteq Z(R)$ for any proper endomorphism $\p$.
\end{Prop}
\begin{proof}
For $\p\in \End_c(R)$, we put $\p(R)=Re_{\p}$ with $e_{\p}\in E(Z(R))$.  First we claim that if $\p\in \End_c(R)$ and $z\in Z(R)$, then $\p(z)\in Z(R)$.  Indeed, if $r\in R$, then $re_{\p}\in Re_{\p}=\p(R)$ and so $re_{\p} = \p(r')$ with $r'\in R$.  Thus $\p(z)r = (\p(z)e_{\p})r = \p(z)(re_{\p})=\p(z)\p(r') = \p(zr')=\p(r'z) = \p(r')\p(z) = re_{\p}\p(z)=r\p(z)$.  So if $\p,\psi\in \End_c(R)$, then $\p(\psi(R)) = \p(Re_{\psi}) = \p(R)\p(e_{\psi}) = Re_{\p}\p(e_{\psi})$ and $e_{\p}\p(e_{\psi})$ is a central idempotent as $\p(Z(R))\subseteq Z(R)$.

Clearly, $\p_e(r)=re$ with $e$ a central idempotent is a proper endomorphism with image $Re$.  Conversely, if $\p\in \End_c(R)$ is an idempotent, then $\p$ fixes $\p(R) = Re_{\p}$.  So if $r\in R$, then $\p(r)=\p(r)e_{\p}=\p(r)\p(e_{\p}) = \p(re_{\p}) =re_{\p}$, as required.  This completes the proof.
\end{proof}

Since $\End_c(R)$ has commuting idempotents, the von Neumann regular elements of $\End_c(R)$ form an inverse semigroup. Thus it is natural to consider actions of inverse semigroups on rings by proper endomorphisms.

\begin{Def}
We define an action of an inverse semigroup $S$  on a ring $R$ to be a homomorphism $\alpha\colon S\to \End_c(R)$, written $s\mapsto \alpha_s$.  Often we write $sr$ for $\alpha_s(r)$.  If $e\in E(S)$, let $1_e$ be the central idempotent with $\alpha_e(R) = R1_e$.    We say that action is \emph{non-degenerate} if $\alpha$ is idempotent separating and $R=\sum_{e\in E(S)}R1_e$.
\end{Def}

A set $E$ of idempotents of a ring $R$ is called a \emph{set of local units} if, for all finite subsets $F$ of $R$, there is an idempotent $e\in E$ with $F\subseteq eRe$.
If $S$ admits a non-degenerate action on $R$, then it is easy to check that the Boolean algebra generated by the $1_e$ with $e\in E(S)$ is a set of local units for $R$.  If $S$ is a Boolean inverse semigroup and if $\alpha$ is additive, then the set of $1_e$ with $e\in E(S)$ already is a set of local units.

\begin{Prop}\label{p:where.it.lies}
Let $S$ have a non-degenerate action $\alpha$ on $R$.
\begin{enumerate}
\item $\alpha_s(R) = R1_{ss^*}$.
\item  $\alpha_s(R1_e) = R1_{ses^*}$.
\end{enumerate}
\end{Prop}
\begin{proof}
We have that $\alpha_s(R) = \alpha_{ss^*}(\alpha_s(R))\subseteq R1_{ss*}$.  On the other hand, since $\alpha_{ss^*}(r) = r1_{ss^*}$  by Proposition~\ref{p:proper.endos}, we have \[r1_{ss^*} = \alpha_{ss^*}(r) = \alpha_s(\alpha_{s^*}(r))\in \alpha_s(R).\]  This establishes the first item.  The second follows from the first because $\alpha_s(R1_e) = \alpha_s\alpha_e(R)=\alpha_{se}(R)=R1_{ses^*}$.
\end{proof}

We typically are interested in the case that $S$ has a zero and $\alpha$ preserves $0$ (so $\alpha_0(r)=0$ for all $r\in R$).  We call such an action \emph{zero-preserving}.

One should note that $\alpha_s$ restricts to an isomorphism $R1_{s^*s}\to R1_{ss^*}$, as is easily checked.

To define a notion of crossed product, we need to next consider twists.  So let $S$ have a non-degenerate action on a ring $R$.  Put $\til R = \bigcup_{e\in E(S)}(Z(R)1_e)^{\times}$.

\begin{Prop}\label{p:unit.sheaf}
The set $\til R$ is a commutative inverse semigroup under multiplication with $E(\til R) =\{1_e\mid e\in E(S)\}$.  Moreover, if $s\in S$, then $\alpha_s(\til R)\subseteq \til R$ and hence $S$ acts on $\til R$ by endomorphisms.
\end{Prop}
\begin{proof}
Given $r,r'\in (Z(R)1_e)^\times$ and $u,u'\in (Z(R)1_f)^\times$ with $rr'=1_e$ and $uu'=1_f$, we have that $ru,u'r'\in Z(R)1_{ef}$ and $ruu'r' = r1_fr'=rr'1_f=1_e1_f=1_{ef}$.  Thus $ru,u'r'\in \til R$ and $(ru)(u'r')(ru)=ru$. Therefore, $\til R\subseteq Z(R)$ is a commutative von Neumann regular semigroup and hence inverse. Clearly, $E(\til R) =\{1_e\mid e\in E(S)\}$.

If $s\in S$ and $r,r'\in (Z(R)1_e)^\times$ with $rr'=1_e$, then $\alpha_s(r)\alpha_s(r') = \alpha_s(1_e) = 1_{ses^*}$, and so $\alpha_s(r)\in (Z(R)1_{ses^*})^\times$ (using Propositions~\ref{p:proper.endos} and~\ref{p:where.it.lies}).
\end{proof}

We now have that if the action is non-degenerate, then $\til R$ is an $S$-module where we put $p(r)=e$ if $r\in (Z(R)1_e)^\times$ by Proposition~\ref{p:where.it.lies}(2).

\subsection{Crossed products}
We continue to work in the context of an inverse semigroup $S$ with a non-degenerate action $\alpha\colon S\to \End_c(R)$ on $R$.  Fix a  normalized $2$-cocycle $c\colon S\times S\to \til R$.
To define the crossed product $R\rtimes_{\alpha,c} S$, we proceed in two steps.

\begin{Prop}\label{p:define.crossed.prod}
Let $S$ be an inverse semigroup  and let $\alpha$ be a  non-degenerate action of $S$ on a ring $R$. Let $c\colon S\times S\to \til R$ be a normalized $2$-cocycle.
\begin{enumerate}
\item The abelian group $\bigoplus_{s\in S} R\delta_s$ (here $\delta_s$ is an indexing symbol) is a ring with product defined by $r\delta_s\cdot r'\delta_t = r(sr')c(s,t)\delta_{st}$.
\item The additive subgroup $I$ generated by $r\delta_s-rc(t,s^*s)^*\delta_t$ with $s\leq t$ is a two-sided ideal.  If $S$ has a zero and the action is zero-preserving, then $I$ contains $R\delta_0$.
\end{enumerate}
\end{Prop}
\begin{proof}
For the first item, it suffices to check associativity.  We have
\begin{align*}
((r_1\delta_s)(r_2\delta_t))r_3\delta_u &= r_1(sr_2)c(s,t)(str_3)c(st,u)  \delta_{stu}, \\
r_1\delta_s(r_2\delta_t r_3\delta_u) &= r_1s(r_2(tr_3)c(t,u))c(s,tu)\delta_{stu}\\ &= r_1(sr_2)(str_3)(sc(t,u))c(s,tu)\delta_{stu}.
\end{align*}
  Since $c$ takes values in the center of $R$, associativity follows from $c$ being a $2$-cocycle.

The second item is more technical.    Consider $s\leq t\in S$ and $r\in R$.  If $u\in S$ and $a\in R$, then
\[a\delta_u(r\delta_s-rc(t,s^*s)^*\delta_t)= a(ur)c(u,s)\delta_{su} - a(ur)(uc(t,s^*s)^*)c(u,t)\delta_{ut}.\]  We have $su\leq ut$ and so it suffices to show that $c(u,s)c(ut, s^*u^*us)^* = (uc(t,s^*s)^*)c(u,t)$, which is the content of Proposition~\ref{p:props.normalized.cocycle}(10).
On the other hand,
\[(r\delta_s-rc(t,s^*s)^*\delta_t)a\delta_u = r(sa)c(s,u)\delta_{su}-rc(t,s^*s)^*(ta)c(t,u)\delta_{tu}.\]  Note that $c(t,s^*s)\in R1_{ss^*}$ and so $c(t,s^*s)^*(ta) = c(t,s^*s)^*1_{ss^*}(ta) = c(t,s^*s)^*(ta)1_{ss^*} = c(t,s^*s)^*(ss^*(ta)) = c(t,s^*s)^*(sa)$.  Thus it suffices to show that
$r(sa)c(s,u)\delta_{su}-r(sa)c(t,s^*s)^*c(t,u)\delta_{tu}\in I$.  Since $su\leq tu$, it suffices to show that $c(s,u)c(tu,u^*s^*su) = c(t,s^*s)^*c(t,u)$, which follows from Proposition~\ref{p:props.normalized.cocycle}(11).  This completes the proof that $I$ is an ideal.

 If $\alpha$ is zero preserving, then since $0\leq 0$, we have that $r\delta_0-rc(0,0)^*\delta_0\in I$. But $c(0,0)=0$ as $1_0=0$ and so $(Z(R)1_0)^\times =0$. Thus  $r\delta_0\in I$.
\end{proof}

We can now define the crossed product.
\begin{Def}
Let $R$ be a ring, $S$ an inverse semigroup and $\alpha\colon S\to \End_c(R)$ a non-degenerate action.  Let $c\colon S\times S\to \til R$ be a normalized $2$-cocycle.  Then the \emph{crossed product} $R\rtimes_{\alpha,c} S$ is the ring $\bigoplus_{s\in S}R\delta_s/I$ where we retain the notation of Proposition~\ref{p:define.crossed.prod}.
\end{Def}

It is relatively straightforward to verify that two cohomologous normalized $2$-cocycles yield isomorphic crossed products.
\begin{Prop}\label{p:same.crossed}
Let $\alpha\colon S\to \End_c(R)$ be a non-degenerate action of $S$ on $R$.  Let $c,c'\colon S\times S\to \til R$ be normalized two-cocycles.  If $c$ and $c'$ are cohomologous, then $R\rtimes_{\alpha,c} S\cong R\rtimes_{\alpha,c'} S$.
\end{Prop}
\begin{proof}
Let $F\colon S\to \til R$ be a mapping with $F(s)\in (Z(R)1_{ss^*})^\times$, for all $s\in S$, and with $c=c'\cdot \delta F$.  So $c(s,t) = c'(s,t)F(s)(sF(t))F(st)^*$.  Since $c,c'$ are normalized, $c(e,e)=1_e=c'(e,e)$ for $e\in E(S)$.  Thus  we have that \[F(e)= c'(e,e)F(e)(eF(e))F(ee)^* = c(e,e)=1_e\] and so $F$ is idempotent preserving.

 We put $A=\bigoplus_{s\in S}R\delta_s$ with product $r\delta_s\cdot r'\delta_t = r(sr')c(s,t)\delta_{st}$ and let $I$ be the additive subgroup generated by all $r\delta_s-rc(t,s^*s)^*\delta_t$ with $s\leq t$.  Similarly, we let $A'=\bigoplus_{s\in S}R\delta_s$ with multiplication $r\delta_s\cdot r'\delta_t = r(sr')c'(s,t)\delta_{st}$ and $I'$ be the additive subgroup  generated by all $r\delta_s-rc'(t,s^*s)^*\delta_t$ with $s\leq t$.  We first construct a homomorphism $\Phi\colon A\to A'$ with $\Phi(I)\subseteq I'$.  Define $\Phi(r\delta_s) = rF(s)\delta_s$.  Note that $\Phi(r\delta_s\cdot r'\delta_t) = r(sr')c(s,t)F(st)\delta_{st}$, whereas $\Phi(r\delta_s)\Phi(r'\delta_t) = rF(s)\delta_s\cdot r'F(t)\delta_t = r(sr')F(s)(sF(t))c'(s,t)\delta_{st}$.  But $F(st)^*F(st) = 1_{stt^*s^*}$ and $c'(s,t)\in R1_{stt^*s^*}$, whence \[c'(s,t) = F(st)^*F(st)c'(s,t).\]  Therefore, we have that
 \begin{align*}
 \Phi(r\delta_s)\Phi(r'\delta_s) &= r(sr')F(s)(sF(t))F(st)^*c'(s,t)F(st)\delta_{st} \\ &= r(sr')c(s,t)F(st)\delta_{st}=\Phi(r\delta_s\cdot r'\delta_t).
 \end{align*}
   We conclude that $\Phi$ is a homomorphism.

Suppose now that $s\leq t$ and note that $\Phi(r\delta_s-rc(t,s^*s)^*\delta_t) = rF(s)\delta_s - rc(t,s^*s)^*F(t)\delta_t$.  But \[c(t,s^*s)^*=c'(t,s^*s)^*F(t)^*(tF(s^*s))^*F(ts^*s) = c'(t,s^*s)^*F(t)^*1_{ss^*}F(s)\] as $F(s^*s) =1_{s^*s}$ and $ts^*s=s$.  Thus, using $1_{tt^*}\geq 1_{ss^*}$, we have \[rc(t,s^*s)^*F(t)=rF(s)c'(t,s^*s)^*F(t)^*F(t)1_{ss^*} = rF(s)c'(t,s^*s)^*.\]   Therefore, $ rF(s)\delta_s - rc(t,s^*s)^*F(t)\delta_t=rF(s)\delta_s-rF(s)c'(t,s^*s)^*\delta_t\in I'$, and so we have a well-defined homomorphism $\p\colon R\rtimes_{\alpha,c}S\to R\rtimes_{\alpha,c'} S$ given by $\p(r\delta_s+I) = \Phi(r\delta_s)+I' = rF(s)\delta_s+I'$.  Similarly, we have a well-defined homomorphism $\psi\colon R\rtimes_{\alpha,c'}S\to R\rtimes_{\alpha,c} S$ given by $\psi(r\delta_s+I') = rF(s)^*\delta_s+I$.  We show that these two homomorphisms are inverse to each other.  By symmetry, it suffices to consider $\psi\circ \p$.  We have that $\psi\p(r\delta_s+I) = \psi(rF(s)\delta_s+I') = rF(s)F(s)^*\delta_s+I = r1_{ss^*}\delta_s+I$.  But since $s\leq s$ and $c(s,s^*s)= 1_{ss^*}$ by Proposition~\ref{p:props.normalized.cocycle}(1), it follows that $r\delta_s-r1_{ss^*}\delta_s = r\delta_s-rc(s,s^*s)^*\delta_s\in I$.  Thus $\psi\p(r\delta_s+I) = r\delta_s+I$.  This completes the proof.
\end{proof}

 When the $2$-cocycle $c$ is trivial, then $R\rtimes_{\alpha,c} S$ is called a \emph{skew inverse semigroup ring}.  Although it looks superficially different, this notion of a skew inverse semigroup ring coincides with the standard one for what are called spectral actions in~\cite{Skewasconv}.  Note that when $c$ is trivial, $I$ is generated by all differences $r\delta_s-r1_{ss^*}\delta_t$ with $s\leq t$.

The following two results yield a partial generalization of a result for skew inverse semigroup rings~\cite[Proposition~3.1]{simpleskew}.

\begin{Prop}\label{p:diagonalsub}
Let $\alpha\colon S\to \End_c(R)$ be a non-degenerate action and $c\colon S\times S\to \til R$ a normalized $2$-cocycle.  Then the additive subgroup of the ring $R\rtimes_{\alpha,c} S$ generated by the elements $r\delta_e+I$ with $r\in R$ and $e\in E(S)$ is a subring isomorphic to a quotient of $R$.   If there is an additive group homomorphism $\tau\colon R\rtimes_{\alpha,c}S\to R$ with $\tau(r\delta_e+I) = r$ for all $r\in R1_e$ and $e\in E(S)$, then this subring is isomorphic to $R$.
\end{Prop}
\begin{proof}
Let $A$ be the additive subgroup generated by the $r\delta_e+I$ with $r\in R$ and $e\in E(S)$.  Since $r\delta_e\cdot r'\delta_f+I = r(er')c(e,f)\delta_{ef}+I$, it is clear that $R$ is a subring.  Suppose that $r\in R$.  Since the action is non-degenerate, we may write $r=\sum_{i=1}^n r_i$ with $r_i\in R1_{e_i}$ and $e_i\in E(S)$, for $i=1,\ldots, n$.  Define $\rho(r) = \sum_{i=1}^n r_i\delta_{e_i}+I$.  We claim that $\rho$ is a well-defined  homomorphism from $R$ to $A$.  To show that $\rho$ is well defined, it suffices to show that if $0=\sum_{i=1}^n r_i$ with $r_i\in R1_{e_i}$, then $\sum_{i=1}^nr_i\delta_{e_i}\in I$.  We proceed by induction on $n$, with the case $n=1$ being trivial, as then $r_1=0$ and so $r_1\delta_{e_1}=0\in I$.  Suppose now that the result is true for $n-1$ and  $0=\sum_{i=1}^n r_i$ with $r_i\in R1_{e_i}$ with $n\geq 2$.  Then $r_n = -\sum_{i=1}^{n-1}r_i$ and so $r_n=r_n1_{e_n} = -\sum_{i=1}^{n-1}r_i1_{e_n}$.  Note that $r_i-r_i1_{e_n}\in R1_{e_i}$ and $0=\sum_{i=1}^{n-1} (r_i-r_i1_{e_n})$.  By induction, we deduce that $\sum_{i=1}^{n-1} (r_i-r_i1_{e_n})\delta_{e_i}\in I$, i.e.,
\begin{equation}\label{eq:embed.eq}
\sum_{i=1}^{n-1}r_i\delta_{e_i}+I = \sum_{i=1}^{n-1}r_i1_{e_n}\delta_{e_i}+I.
\end{equation}
 Since $e_ie_n\leq e_i$, we have that $r_i\delta_{e_ie_n}-r_i1_{e_n}\delta_{e_i}\in I$, as $r_ic(e_i, e_ie_n)^* = r_i1_{e_i}1_{e_n}=r_i1_{e_n}$ by Proposition~\ref{p:props.normalized.cocycle}(4).  On the other hand, since $e_ie_n\leq e_n$, we have that $r_i\delta_{e_ie_n} - r_i1_{e_n}\delta_{e_n}\in I$ as $r_ic(e_n,e_ie_n)^* = r_i1_{e_i}1_{e_n} = r_i1_{e_n}$.   Therefore, $r_i1_{e_n}\delta_{e_i}+I=r_i1_{e_n}\delta_{e_n}+I$.   Thus, we deduce from \eqref{eq:embed.eq} that
 \[\sum_{i=1}^{n-1}r_i\delta_{e_i}+I = \sum_{i=1}^{n-1}r_i1_{e_n}\delta_{e_n}+I  = -r_n\delta_{e_n}+I. \]
 It follows that $\sum_{i=1}^nr_i\delta_{e_i}\in I$, as required.

To verify that $\rho$ is a homomorphism, let $r=\sum_{i=1}^n r_i$ and $r'=\sum_{j=1}^m r'_j$ with $r_i\in R1_{e_i}$ and $r_j\in R1_{f_j}$ with $e_i, f_j\in E(S)$ for $1\leq i\leq n$, $1\leq j\leq m$.  Then $rr' = \sum_{i=1}^n\sum_{j=1}^m r_ir'_j$ and $r_ir'_j\in R1_{e_if_j}$. Therefore, $\rho(rr') = \sum_{i=1}^n\sum_{j=1}^m r_ir'_j\delta_{e_if_j}+I$.  On the other hand, \[\rho(r)\rho(r') =\sum_{i=1}^nr_i\delta_{e_i}\sum_{j=1}^mr_j\delta_{f_j}+I=\sum_{i=1}^n\sum_{j=1}^m r_i(e_ir_j)c(e_i,f_j)\delta_{e_if_j}+I.\]  But $e_ir_j = r_j1_{e_i}$ and $c(e_i,f_j) =1_{e_if_j}$ by Proposition~\ref{p:props.normalized.cocycle}(4) and so \[r_i(e_ir_j)c(e_i,f_j)=r_ir_j1_{e_i}1_{e_if_j} = r_ir_j.\]  Therefore, $\rho$ is a homomorphism.


To check that $\rho$ is surjective, notice that since $e\leq e$, we have that $r\delta_e+I = rc(e,e)^*\delta_e+I = r1_e\delta_e+I = \rho(r1_e)$.   
Suppose that $\tau$, as in the final statement, exists.  If $r\in R$ with $r=\sum_{i=1}^n r_i$ with $r_i\in R1_{e_i}$, for $i=1,\ldots, n$, then $\tau(\rho(r)) = \tau\left(\sum_{i=1}^n r_i\delta_{e_i}+I\right)=\sum_{i=1}^n r_i = r$ and so $\rho$ is injective.
\end{proof}

We suspect that $\rho$ is injective in general, but we could only prove the existence of the additive homomorphism $\tau$ in certain cases.

\begin{Prop}\label{p:good.cases}
Let $\alpha\colon S\to \End_c(R)$ be a non-degenerate action and $c\colon S\times S\to \til R$ a normalized $2$-cocycle.  Then an additive group homomorphism $\tau\colon R\rtimes_{\alpha,c}S\to R$ with $\tau(r\delta_e+I) = r$ for all $r\in R1_e$ exists in the following cases.
\begin{enumerate}
\item If $c(t,s^*s)=1_{ss^*}$ whenever $s\leq t$ (i.e., if $c$ is trivial, that is, $R\rtimes_{\alpha,c} S$ is a skew inverse semigroup ring).
\item If, for each $s\in S$, either the set $E(s)=\{e\in E(S)\mid e\leq s\}$ is empty or has a unique maximum element $e(s)$.
\end{enumerate}
\end{Prop}
\begin{proof}
Suppose the first item holds.  Define a homomorphism of additive groups $\tau'\colon \bigoplus_{s\in S}R\delta_s\to R$ by $r\delta_s\mapsto r1_{ss^*}$.  Then a generator for $I$ is of the form $r\delta_s-r1_{ss^*}\delta_t$ with $s\leq t$, as $c(t,s^*s)^*=1_{ss^*}$,  and $\tau'(r\delta_s-r1_{ss^*}\delta_t) = r1_{ss^*}-r1_{ss^*}1_{tt^*} =0$ since $s\leq t$.  Therefore, $\tau'$ induces an additive group homomorphism $\tau\colon R\rtimes_{\alpha,c}S\to R$. Moreover, $\tau(r\delta_e+I) = r1_e=r$ for $r\in R1_e$ with  $e\in E(S)$.

Now, suppose that each set $E(s)$ is either empty or has a maximum $e(s)$.  Define a homomorphism of additive groups $\tau'\colon \bigoplus_{s\in S}R\delta_s\to R$ by \[\tau'(r\delta_s) = \begin{cases} rc(s,e(s)), & \text{if}\ E(s)\neq \emptyset \\ 0, & \text{else.}\end{cases} \]  We need to check that $I\subseteq \ker \tau'$.  Suppose that $s\leq t$. Obviously $E(s)\subseteq E(t)$.  On the other hand, if $e\in E(t)$, then $es^*s\leq ts^*s=s$ and so $es^*s\in E(s)$.  Thus $E(s)=\emptyset$ if and only if $E(t)=\emptyset$.  In particular, if $s\leq t$ and $E(s)=\emptyset=E(t)$, then $\tau'(s\delta_s-sc(t,s^*s)^*\delta_t) =0$.  So let us assume that $E(s)$ and $E(t)$ are non-empty.  Then $e(s)\leq s\leq t$ implies $e(s)\leq e(t)$ and $e(s)\leq s^*s$.  Thus $e(s)\leq e(t)s^*s$.  But $e(t)s^*s\leq ts^*s=s$, and so $e(t)s^*s\leq e(s)$. Therefore, $e(t)s^*s=e(s)$.  A similar argument shows that $ss^*e(t)=e(s)$.

We now compute $(tc(s^*s,e(s)))c(t,s^*se(s)) = c(t,s^*s)c(ts^*s,e(s))$, which yields $c(t,e(s))=c(t,s^*s)c(s,e(s))$  since $c$ is normalized (using Proposition~\ref{p:props.normalized.cocycle}(4)).  But \[(tc(e(t),s^*s))c(t,e(t)s^*s)=c(t,e(t))c(te(t),s^*s) = c(t,e(t))c(e(t),s^*s).\]  As $e(t)s^*s=e(s)$, we deduce, using Proposition~\ref{p:props.normalized.cocycle}(4), that $c(t,e(s))=c(t,e(t))1_{e(s)}=e(s)c(t,e(t))$.  Therefore, $e(s)c(t,e(t))=c(t,s^*s)c(s,e(s))$. Also note that $p(c(t,e(t))1_{ss^*})= ss^*e(t)=e(s)$ and so $c(t,e(t))1_{ss^*}=e(s)(c(t,e(t))1_{ss^*}) =e(s)(ss^*c(t,e(t))) = e(s)c(t,e(t))= c(t,s^*s)c(s,e(s))$.
Thus we have
\begin{align*}
\tau'(r\delta_s-rc(t,s^*s)^*\delta_t)& = rc(e,s)-rc(t,s^*s)^*c(t,e(t))\\ &= rc(e,s)-rc(t,s^*s)^*1_{ss^*}c(t,e(t))  \\&= rc(e,s)-rc(t,s^*s)^*c(t,s^*s)c(s,e(s))\\ &=0
\end{align*}
 as $ss^*\geq e(s)$.  Therefore, $\tau'$ induces a well defined map $\tau\colon R\rtimes_{\alpha,c}S\to R$ with $\tau(r\delta_e+I)=\tau'(r\delta_e) =rc(e,e)=r1_e=r$ if $r\in R1_e$, as $c$ is normalized.
\end{proof}

In addition to the case that $c$ is trivial, we note that if \eqref{eq:extension} is an extension of inverse semigroups which admits an order-preserving and idempotent-section $j\colon S\to T$, then the corresponding normalized $2$-cocycle $c$ satisfies $c(t,s^*s)=1_{ss^*}$ when $s\leq t$ since $j(t)j(s^*s)=j(ts^*s)=j(s)$ by Lemma~\ref{l:donsig}.
Recall that an inverse semigroup $S$ is $E$-unitary if $s\geq e$ with $e\in E(S)$ implies $s\in E(S)$ and $S$ is $0$-$E$-unitary if $S$ has a zero and $s\geq e\neq 0$ with $e\in E(S)$ implies $s\in E(S)$; see~\cite{Lawson} for details.
Examples of inverse semigroups  satisfying the second condition of Proposition~\ref{p:good.cases} include  $E$-unitary and $0$-$E$-unitary inverse semigroups  and the inverse semigroup of compact open bisections of a Hausdorff ample groupoid $\mathscr G$. 

We now establish a universal property for the crossed product, generalizing that of skew inverse semigroup rings~\cite{Skewasconv}.  Let $S$ be an inverse semigroup with a non-degenerate action $\alpha\colon S\to \End_c(R)$ on a ring $R$ and let $c\colon S\times S\to \til R$ be a normalized $2$-cocycle.  Then a \emph{covariant representation} of $(\alpha,c)$ in a ring $A$ consists of a ring homomorphism $\rho\colon R\to A$ and a map $\psi\colon S\to A$ such that:
\begin{itemize}
\item [(C1)] $\psi(s)\rho(r) = \rho(sr)\psi(s)$;
\item [(C2)] $\rho(1_e) = \psi(e)$;
\item [(C3)] $\psi(s)\psi(t) = \rho(c(s,t))\psi(st)$; and
\item [(C4)] $\psi(ss^*)\psi(s) =\psi(s)$
\end{itemize}
for all $s,t\in S$, $r\in R$ and $e\in E(S)$.  Note that (C4) is equivalent to
\begin{itemize}
\item [(C4')] $\psi(s)\psi(s^*s)=\psi(s)$
\end{itemize}
in the presence of (C1), (C2).
Indeed, \[\psi(s)\psi(s^*s)=\psi(s)\rho(1_{s^*s})=\rho(1_{ss^*})\psi(s)=\psi(ss^*)\psi(s)=\psi(s)\] by (C1), (C2) and (C4).  So (C4) implies (C4') and the reverse implication is dual.  In the case that $c$ is trivial, the conjunction of (C3) and (C4) is equivalent to $\psi$ being a semigroup homomorphism, and so the notion of covariant representation for skew inverse semigroup rings given here coincides with that considered in~\cite{Skewasconv}.

\begin{Prop}\label{p:universal.prop}
Let $S$ be an inverse semigroup with a non-degenerate action $\alpha\colon S\to \End_c(R)$ on a ring $R$ and $c\colon S\times S\to \til R$ a normalized $2$-cocycle. Let $\rho\colon R\to R\rtimes_{\alpha,c} S$ be the homomorphism from the proof of Proposition~\ref{p:diagonalsub} and let $\psi\colon S\to R\rtimes_{\alpha,c} S$ be given by $\psi(s) = 1_{ss^*}\delta_s+I$.  Then $(\rho,\psi)$ is a covariant representation and it is the universal covariant representation, that is, if $\rho'\colon R\to A$ and $\psi'\colon S\to A$ give a covariant representation, then there is a unique homomorphism $\pi\colon R\rtimes_{\alpha,c}S\to A$ such that $\rho'=\pi\rho$ and $\psi'=\pi\psi$.  Namely, one has  $\pi(r\delta_s+I) = \rho'(r)\psi'(s)$.
\end{Prop}
\begin{proof}
We first check that $\rho,\psi$ yield a covariant representation.  If $r\in R$ with $r=\sum_{i=1}^nr_i$ with $r_i\in R1_{e_i}$, then $\rho(r)=\sum_{i=1}^nr_i\delta_{e_i}+I$ and so
\begin{equation}\label{eq:covariant.eq1}
\psi(s)\rho(r) = \sum_{i=1}^n1_{ss^*}\delta_s \cdot r_i\delta_{e_i}+I = \sum_{i=1}^n(sr_i)c(s,e_i)\delta_{se_i}+I.
\end{equation}
One the other hand, $sr = \sum_{i=1}^nsr_i$ and $sr_i\in R1_{se_is^*}$.  Therefore,
\begin{equation}\label{eq.covariant.eq2}
\rho(sr)\psi(s) = \sum_{i=1}^nsr_i\delta_{se_is^*}\cdot 1_{ss^*}\delta_s+I = \sum_{i=1}^n (sr_i)c(se_is^*,s)\delta_{se_i}+I.
\end{equation}
It follows that the right hand sides of \eqref{eq:covariant.eq1} and \eqref{eq.covariant.eq2} are equal by Proposition~\ref{p:props.normalized.cocycle}(9), yielding (C1).  For (C2), note that $\rho(1_e) = 1_e\delta_e+I = \psi(e)$ by definition.  For (C3), we compute
\begin{align*}
\psi(s)\psi(t) &= 1_{ss^*}\delta_s\cdot 1_{tt^*}\delta_t+I = 1_{ss^*}(s1_{tt^*})c(s,t)\delta_{st}+I\\ &= c(s,t)1_{stt^*s^*}\delta_{st}+I= c(s,t)\delta_{stt^*s^*}\cdot 1_{stt^*s^*}\delta_{st}+I\\ &= \rho(c(s,t))\psi(st)
\end{align*}
since $c(s,t)\in R1_{stt^*s^*}$.  Here we have used Proposition~\ref{p:where.it.lies} to obtain $s(1_{tt^*}) = 1_{stt^*s^*}$ and we have used Proposition~\ref{p:props.normalized.cocycle}(1) to get $c(stt^*s^*,st)=1_{stt^*s^*}$.  This verifies (C3).  Finally, (C4) follows because $\psi(ss^*)\psi(s) = 1_{ss^*}\delta_{ss^*}\cdot 1_{ss^*}\delta_s +I= 1_{ss^*}\delta_s+I$ by Proposition~\ref{p:props.normalized.cocycle}(1).

Suppose now that $\rho'\colon R\to A$ and $\psi'\colon S\to A$ give rise to a covariant representation.  First we check that $r\delta_s\mapsto \rho'(r)\psi'(s)$ yields a well-defined homomorphism $\pi'\colon \bigoplus_{s\in S}R\delta_s\to A$.  Indeed,
\begin{align*}
\pi'(r_1\delta_{s_1})\pi'(r_2\delta_{s_2}) & = \rho'(r_1)\psi'(s_1)\rho'(r_2)\psi'(s_2) = \rho'(r_1)\rho'(s_1r_2)\psi'(s_1)\psi'(s_2)\\ &= \rho'(r_1(s_1r_2)c(s_1,s_2))\psi'(s_1s_2) = \pi'(r_1\delta_{s_1}\cdot r_2\delta_{s_2})
\end{align*}
 by (C1) and (C3).

 We now check that $I\subseteq\ker \pi'$.  Let $s\leq t$ and $r\in R$.
First note that $ts^*st^* = ss^*$ and so $c(t, s^*s)=c(ss^*,t)=c(ss^*,t)1_{ss^*}$ by Proposition~\ref{p:props.normalized.cocycle}(9).
Therefore,
\begin{align*}
\pi'(rc(t,s^*s)^*\delta_t) &= \rho'(rc(ss^*,t)^*1_{ss^*})\psi'(t) = \rho'(rc(ss^*,t)^*)\rho'(1_{ss^*})\psi'(t)\\ & =  \rho'(rc(ss^*,t)^*)\psi'(ss^*)\psi'(t)=\rho'(rc(ss^*,t)^*c(ss^*,t))\psi'(s)\\ &=\rho'(r1_{ss^*})\psi'(s) = \rho'(r)\rho'(1_{ss^*})\psi'(s)\\ &=\rho'(r)\psi'(ss^*)\psi'(s)= \rho'(r)\psi'(s)\\ &= \pi'(r\delta_s)
\end{align*}
 by (C2), (C3) and (C4).  Thus $I\subseteq\ker \pi'$ and so $\pi\colon R\rtimes_{\alpha,c} S\to A$ given by $\pi(r\delta_s+I) = \rho'(r)\psi'(s)$ is well defined. If $r\in R$ with $r=\sum_{i=1}^n r_i$ where $r_i\in R1_{e_i}$, then $\pi(\rho(r)) = \sum_{i=1}^n\pi'(r_i\delta_{e_i}) =\sum_{i=1}^n\rho'(r_i)\psi'(e_i) = \sum_{i=1}^n\rho'(r_i)\rho'(1_{e_i}) =\rho'(r)$ using (C2).  Also \[\pi(\psi(s)) = \pi'(1_{ss^*}\delta_s) = \rho'(1_{ss^*})\psi'(s) = \psi'(ss^*)\psi'(s)=\psi'(s)\] using (C2) and (C4).

 For uniqueness, if $\gamma\colon R\times_{\alpha,c} S\to A$ is a homomorphism with $\gamma\rho=\rho'$ and $\gamma\psi=\psi'$, then since $r\delta_s+I = r1_{ss^*}\delta_s+I$ as $s\leq s$ and $c(s,s^*s)=1_{ss^*}=c(ss^*,s)$ by Proposition~\ref{p:props.normalized.cocycle}(1), we have that
 \begin{align*}
 \gamma(r\delta_s+I) &= \gamma(r1_{ss^*}\delta_s+I) = \gamma(r1_{ss^*}\delta_{ss^*}\cdot 1_{ss^*}\delta_s+I)\\ &=\gamma(\rho(r1_{ss^*}))\gamma(\psi(s)) = \rho'(r1_{ss^*})\psi'(s)\\  &=\rho'(r)\psi'(ss^*)\psi'(s)=\rho'(r)\psi'(s)
 \end{align*}
 by (C2) and (C4), as required.  This completes the proof.
\end{proof}

We note that our notion of crossed product seems related to, but different from, the $C^*$-algebraic twisted action crossed product in~\cite{ExelBuss}. In their notion, the action of an element of the inverse semigroup on a ring is only partially defined.  Also, they loosen the requirement on how the action of a product of semigroup elements behaves.  We could also allow greater generality by letting $\alpha\colon S\to \End_c(R)$ just be a map, not a homomorphism,  and allowing the $2$-cocycle $c$ to take values in the non-commutative inverse semigroup $\bigcup_{e\in E(S)}(R1_e)^\times$, but the conditions on $\alpha$ and $c$ that yield associativity are slightly more complicated and we shall not need this more general construction here in any event as we shall primarily be interested in the case that $R$ is commutative, in which case the general construction will reduce to ours.

\subsection{Twisted Steinberg algebras}
Our goal is to show that twisted Steinberg algebras of Hausdorff ample groupoids are inverse semigroup crossed products, generalizing the result of~\cite{GonRoy2017a} for the case of untwisted Steinberg algebras and skew inverse semigroup rings.  Let $\mathscr G$ be a Hausdorff ample groupoid and let
\begin{equation*}
R^\times \times \mathscr G\skel 0\xrightarrow{\,\,\iota\,\,}\Sigma\xrightarrow{\,\,\p\,\,} \mathscr G
\end{equation*}
be a discrete $R^\times$-twist.  The associated twisted Steinberg algebra $A_R(\mathscr G;\Sigma)$ consists of all locally constant functions $f\colon \Sigma\to R$ which are $R^\times$-anti-equivariant, in the sense that $f(rs) = r\inv f(s)$ for $s\in \Sigma$ and $r\in R^\times$, and have compact support modulo $R^\times$, that is, $\p(\supp(f))$ is compact.  Addition and the $R$-module structure are pointwise.  The product is defined as follows.  Choose a set theoretic section $p\colon \mathscr G\to \Sigma$ (we do not assume continuity of $p$).  Then the convolution is given by
\[f\ast f'(s) = \sum_{\ran(g)=\ran(\p(s)))}f(p(g))f'(p(g)\inv s).\]  This does not depend on the choice of $p$.  Note that we are following the conventions of~\cite{twistedreconstruction} and not~\cite{twists}, which uses equivariant functions.

Let $A=C_c(\mathscr G\skel 0,R)$ be the ring of compactly supported locally constant mappings $\mathscr G\skel 0\to R$ with pointwise operations.  There is a natural zero-preserving action $\alpha\colon \Gamma_c(\mathscr G)\to \End_c(A)$ given by
\[\alpha_U(f)(x) =  \begin{cases} f(\dom((\ran|_U)\inv(x))), & \text{if}\ x\in \ran(U)\\ 0, & \text{else.}\end{cases}\]  Notice that the image of $\alpha_U$ is $A1_{\ran(U)}$ and that the action is non-degenerate and additive.  Moreover, $\til A = C_c(\mathscr G\skel 0,R^\times\cup \{0\})$ with the action we previously considered in Proposition~\ref{p:ismodule}.  (That $\alpha$ is an action is essentially the same proof as in Proposition~\ref{p:ismodule}.)  Hence, we can consider a normalized $2$-cocycle $c$ associated to the twist as per Theorem~\ref{t:classify} and form the crossed product $A\rtimes_{\alpha,c} \Gamma_c(\mathscr G)$ and the crossed product ring is independent of the choice of $2$-cocycle up to isomorphism. More precisely, we fix a set theoretic section $j\colon \Gamma_c(\mathscr G)\to \Gamma_c(\Sigma)$ with $j|_{E(\Gamma_c(\mathscr G))} = (\Gamma_c(\p)|_{E(\Gamma_c(\Sigma))})\inv$.  One then has
\begin{align*}
j(U)j(V) &= \iota(\{(c(U,V)(x),x)\mid x\in \ran(UV)\})j(UV) \\ &= \{c(U,V)(\ran(\p(s)))s\mid s\in j(UV)\}.
\end{align*}
 Our goal it to show that the crossed product $A\rtimes_{\alpha,c} \Gamma_c(\mathscr G)$ is isomorphic to $A_R(\mathscr G;\Sigma)$.  Note that $\bigoplus_{U\in \Gamma_c(\mathscr G)}A\delta_U$ and $A\rtimes_{\alpha,c} \Gamma_c(\mathscr G)$ are $R$-algebras with $r(a\delta_U) = (ra)\delta_U$ and the corresponding induced $R$-algebra structure on the crossed product, as is easily checked.  Note that Proposition~\ref{p:good.cases} and  Proposition~\ref{p:diagonalsub} apply in this setting to embed $A$ into the crossed product since $\mathscr G$ is Hausdorff.

We proceed by first establishing a number of lemmas.  We retain the notation $I$ for the two-sided ideal in Proposition~\ref{p:define.crossed.prod}(2).

\begin{Lemma}\label{l:support}
Let $a\in A$ and $U\in \Gamma_c(\mathscr G)$.   Let $V=\supp(a)\cap \ran(U)$.  Then $a\delta_U+I = ac(U,\dom(VU))\delta_{VU} +I$ and $\supp(ac(U,\dom(VU))=\ran(VU)$.
\end{Lemma}
\begin{proof}
We have that $U\leq U$ and so $a\delta_U+I = ac(U,\dom(U))^*\delta_U+I = a1_{\ran(U)}\delta_U+I$ by Proposition~\ref{p:props.normalized.cocycle}(1), as $c$ is normalized. Then $a1_{\ran(U)} = a1_V=a1_{\ran(VU)}$ by definition of $V$.  Since $VU\leq U$ we have that
\begin{align*}
ac(U,\dom(VU))\delta_{VU}+I &= ac(U,\dom(VU))c(U,\dom(VU))^*\delta_{U}+I \\ &= a1_{\ran(VU)}\delta_U+I= a\delta_U+I,
\end{align*}
 by the above.  Since $\supp(c(U,\dom(VU)))=\ran(VU) \subseteq \supp(a)$, we have that $\supp(ac(U,\dom(VU))=\ran(VU)$.
\end{proof}

Thus $A\rtimes_{\alpha,c}\Gamma_c(\mathscr G)$ is spanned by elements of the form $a\delta_U+I$ with $\supp(a)=\ran(U)$.

\begin{Lemma}\label{l:disjoint.union}
Suppose that $U\in \Gamma_c(\mathscr G)$ is a union of pairwise disjoint compact open bisections $U_1,\ldots, U_k$ and $\supp(a)=\ran(U)$.  Then $a\delta_U+I = \sum_{i=1}^k a_i\delta_{U_i}+I$ with $\supp(a_i) = \ran(U_i)$ for $i=1,\ldots, k$.
\end{Lemma}
\begin{proof}
Let $a_i = ac(U,\dom(U_i))$. Then $\supp(a_i)=\supp(a)\cap \supp(c(U,\dom(U_i))) = \supp(a)\cap \ran(U_i) = \ran(U_i)$ as $U_i =U\cdot \dom(U_i)$.  Next observe that since $U_i\leq U$, we have that $a_i\delta_{U_i}+I =ac(U,\dom(U_i))\delta_{U_i}+I =ac(U,\dom(U_i))c(U,\dom(U_i))^*\delta_U+I = a1_{\ran(U_i)}\delta_U+I$ and so $\sum_{i=1}^k a_i\delta_{U_i}+I = \sum_{i=1}^k a1_{\ran(U_i)}\delta_U+I = a1_{\ran(U)}\delta_U +I = a\delta_U+I$ as the $\ran(U_i)$ are pairwise disjoint (since $U$ is a bisection) with union $\ran(U)$.
 \end{proof}

Our final lemma finds us a sort of normal form for $A\rtimes_{\alpha,c}\Gamma_c(\mathscr G)$.
\begin{Lemma}\label{l:disjointify}
Every element of $A\rtimes_{\alpha,c}\Gamma_c(\mathscr G)$ is equal to one of the form $\sum_{i=1}^k a_i\delta_{U_i}+I$ where $U_1,\ldots, U_k$ are pairwise disjoint non-empty compact open bisections and $\supp(a_i)= \ran(U_i)$, for $i=1,\ldots, k$ with possibly $k=0$.
\end{Lemma}
\begin{proof}
By Lemma~\ref{l:support}, each element of $A\rtimes_{\alpha,c}\Gamma_c(\mathscr G)$ can be written in the form $\sum_{j=1}^mb_j\delta_{V_j}+I$ with the $V_j\in \Gamma_c(\mathscr G)$ and $\supp(b_j)=\ran(V_j)$.    Since $A\delta_{\emptyset}\subseteq I$ by Proposition~\ref{p:define.crossed.prod}, we may assume that all the $V_j$ are non-empty, unless our element is $0$, in which case there is nothing to prove.   The compact open subsets  of $\mathscr G$ form a Boolean algebra because $\mathscr G$ is Hausdorff.  The Boolean algebra generated by $V_1,\ldots, V_m$ is finite and hence is generated by its atoms $W_1,\ldots, W_k$.   Moreover, since $V=V_1\cup\cdots \cup V_m$ is the maximum of this Boolean algebra, we have $W_i=V\cap W_i= (V_1\cap W_i)\cup\cdots\cup (V_m\cap W_i)$ and hence, since $W_i$ is an atom, we must have $W_i=V_j\cap W_i$ for some $j$, that is, $W_i\subseteq V_j$.  Thus the $W_i$ are compact open bisections.  Each $V_j$ is the disjoint union of the $W_i$ contained in it and so Lemma~\ref{l:disjoint.union} lets us write $b_j\delta_{V_j}+I=\sum_{W_i\subseteq V_j} b_{ji}\delta_{W_i}+I$ with $\supp(b_{ji}) = \ran(W_i)$.  Thus \[\sum_{j=1}^mb_j\delta_{V_j}+I=\sum_{j=1}^m\sum_{W_i\subseteq V_i} b_{ji}\delta_{W_i}+I=\sum_{i=1}^k d_i\delta_{W_i}+I\] with $\supp(d_i)\subseteq \ran(W_i)$.   Using Lemma~\ref{l:support}, we can replace $d_i\delta_{W_i}$ by $a_i\delta_{U_i}$ with $U_i\subseteq W_i$ and $\supp(a_i)=\ran(U_i)$.    As $A\delta_{\emptyset}\subseteq I$ by Proposition~\ref{p:define.crossed.prod}, we may remove also the empty $U_i$.  This completes the proof.
\end{proof}

We are now ready to prove the main theorem of this section.

\begin{Thm}
Let $R$ be a commutative ring, $\mathscr G$ a Hausdorff ample groupoid and $R^\times\times \mathscr G\skel 0\xrightarrow{\,\,\iota\,\,}\Sigma\xrightarrow{\,\,\p\,\,} \mathscr G$ a twist.  Let $j\colon \Gamma_c(\mathscr G)\to \Gamma_c(\Sigma)$ be an idempotent-preserving set theoretic section and let $c\colon \Gamma_c(\mathscr G)\times \Gamma_c(\mathscr G)\to C_c(\mathscr G\skel 0, R^\times \cup\{0\})$ the corresponding normalized $2$-cocycle.  Then the $R$-algebras $C_c(\mathscr G\skel 0,R)\rtimes_{\alpha,c} \Gamma_c(\mathscr G)$ and  $A_R(\mathscr G;\Sigma)$ are isomorphic.
\end{Thm}
\begin{proof}
We retain the above notation.  Define a mapping from $\bigoplus_{U\in \Gamma_c(\mathscr G)}A\delta_U$ (where $A=C_c(\mathscr G\skel 0,R)$) to $A_R(\mathscr G;\Sigma)$ as follows. Let $\pi\colon R^\times\times \mathscr G\skel 0\to R^\times$ be the projection.   For $a\in A$ and $U\in \Gamma_c(\mathscr G)$, put \[f_{a,U}(s) =\begin{cases} a(\ran(\p(s)))(\pi(\iota\inv(s((\p|_{j(U)})\inv(\p(s)))\inv)))\inv, & \text{if}\ \p(s)\in U\\ 0, & \text{else.}\end{cases} \] (This makes sense by Proposition~\ref{p:local.homeo}.)
So $f_{a,U}(s)$ is $0$, unless $\p(s)\in U$, in which case it is $a(\ran(\p(s)))t$ where $t\in R^\times$ with $ts\in j(U)$.
 Since $U$ is clopen, it is obvious that $f_{a,U}$ is continuous (i.e., locally constant).  Notice that $\p(\supp(f_{a,U})) =\supp(a)U$ is compact.  For anti-equivariance, we compute
\[f_{a,U}(ts) =\begin{cases} a(\ran(\p(s)))(\pi(\iota\inv(ts((\p|_{j(U)})\inv(\p(s)))\inv)))\inv, & \text{if}\ \p(s)\in U\\ 0, & \text{else}\end{cases} \] as $\p(ts)=\p(s)$.  Thus $f_{a,U}(ts) =t\inv f_{a,U}(s)$, and so $f_{a,U}\in A_R(\mathscr G;\Sigma)$.

 We claim that $\psi(a\delta_U) = f_{a,U}$ is a homomorphism $\bigoplus_{U\in \Gamma_c(\mathscr G)}A\delta_U$ onto $A_R(\mathscr G;\Sigma)$.  It is clearly $R$-linear.  We first compute $f_{a,U}\ast f_{b,V}(s)$.  We choose a section $p\colon \mathscr G\to \Sigma$ such that $p|_U = (\p|_{j(U)})\inv$.  Then in order for $f_{a,U}\ast f_{b,V}(s)\neq 0$ we need $\p(s)\in UV$.  Suppose this is the case and put $\p(s)=gh$ with $g\in U$ and $h\in V$.  Let $\til g=(\p|_{j(U)})\inv(g)$ and $\til h =(\p|_{j(V)})\inv(h)$.  Put $\til s = (\p|_{j(UV)})\inv(s)$.  Then we can write $s=t\til s$ with $t\in R^\times$  and $\til g\til h = c(U,V)(\ran(\p(s)))\til s$.  We compute that \[f_{a,U}\ast f_{b,V}(s) = a(\ran(g))b(\ran(h))\pi(\iota\inv(\til g\inv s\til h\inv))\inv.\]  But $\til g\inv s\til h\inv = t\til g\inv \til s\til h\inv = tc(U,V)(\ran(\p(s)))\inv\ran(h)$.  Thus  \[f_{a,U}\ast f_{b,V}(s) =a(\ran(g))b(\ran(h))t\inv c(U,V)(\ran(\p(s))).\]  On the other hand, $a\delta_Ub\delta_V = a(Ub)c(U,V)\delta_{UV}$.  Now $f_{a(Ub)c(U,V),UV}(s)$ is $0$ unless $\p(s)\in UV$.  Then, retaining the previous notation, we have
\begin{align*}
  f_{a(Ub)c(U,V),UV}(s) &=  a(\ran(\p(s)))(Ub)(\ran(\p(s)))c(U,V)(\ran(\p(s)))\pi(\iota\inv(s\til s\inv))\inv\\
   &= a(\ran(g))b(\ran(h))c(U,V)(\ran(\p(s)))t\inv
\end{align*}
This completes the proof that $\psi$ is a homomorphism.

If $V\subseteq \Sigma$ is a compact open bisection (whence $\p|_V$ is injective by Propositions~\ref{p:local.homeo}), let $\til{1}_V\in A_R(G;\Sigma)$ be given by
\[\til{1}_V(s) = \begin{cases}\pi(\iota\inv((\p|_V)\inv(\p(s))s\inv)), & \text{if}\ \p(s)\in \p(V)\\ 0, &\text{else,}\end{cases}\] that is, $\til{1}_V$ is supported on $R^\times\cdot V$ and sends $s$ to $t$ if $ts\in V$ with $t\in R^\times$.
 It is shown in~\cite[Proposition~2.8]{twistedreconstruction} that $A_R(\mathscr G;\Sigma)$ is spanned as an $R$-module by the $\tilde{1}_V$.  So to show that $\psi$ is onto, we just need that $\til{1}_V$ is in the image of $\psi$ for any  $V\in \Gamma_c(\Sigma)$.  Let $U=\p(V)\in \Gamma_c(\mathscr G)$ and note that $\p|_{j(U)}\colon j(U)\to U$ is a homeomorphism by Proposition~\ref{p:local.homeo}.
 Then we can define $a\in A$ by
 \[a(x) = \begin{cases} \pi(\iota\inv((\p|_V)\inv ((\ran|_U)^{\inv}(x))(\p|_{j(U)})\inv ((\ran|_U)^{\inv}(x))\inv)), & x\in \ran(U)\\ 0, & \text{else.}\end{cases}\]  Note that $a$ is supported on $\ran(U)$ and so belongs to $A$.
 We claim that $f_{a,U} = \til{1}_V$.     Both functions vanish on all $s\in \Sigma$ with $\p(s)\notin U$.  If $\p(s)\in U$, let $t_1,t_2\in R^\times$ with $t_1s\in j(U)$ and $t_2s\in V$.  Then $\til{1}_V(s)=t_2$.  On the other hand, $f_{a,U}(s) = a(\ran(\p(s)))t_1=t_2t_1\inv t_1=t_2$, as required.  We conclude that $\psi$ is onto.

  We next show that $I=\ker \psi$.
 Suppose that $a\in A$ and $U\leq V\in \Gamma_c(\mathscr G)$.  We need to show that $f_{a,U} = f_{ac(V,\dom(U))^*,V}$.  Note that $\supp(c(V,\dom(U))^*) = \ran(V\cdot \dom(U))=\ran(U)$.  Thus both functions vanish on any $s\in \Sigma$ with $\p(s)\notin U$.  Assume that $\p(s)\in U$.  Let $\til s\in j(U)$ with $\p(\til s)=\p(s)$ and $\ov s\in j(V)$ with $\p(\ov s)=\p(s)$.  Since $j(\dom(U)) = \dom(U)$, we  then have that \[j(V)\dom (U) = \{\iota(c(V,\dom(U))(\ran(g)),\ran(g))\mid g\in U\}j(U),\] and so $\ov s=c(V,\dom(U))(\ran(\p(s)))\til s$. If $s=u\til s$ with $u\in R^\times$.  Then $s=uc(V,\dom(U))(\ran(\p(s)))\inv \ov s$.  Therefore, $f_{a,U}(s) =a(\ran(\p(s)))u\inv$ and
  \begin{align*}
  f_{a(c(V,\dom(U))^*,V}(s) &= a(\ran(\p(s))(c(V,\dom(U))(\ran(\p(s))))\inv c(V,\dom(U))(\ran(\p(s))u\inv\\ &=a(\ran(\p(s))u\inv,
   \end{align*}
  as required.  Thus $I\subseteq \ker \psi$.

Let $z\in \bigoplus_{U\in \Gamma_c(\mathscr G)}A\delta_U$ belong to $\ker \psi$. Then $z+I = \sum_{i=1}^k a_i\delta_{U_i}+I$ with $\supp(a_i)=\ran(U_i)$ and the $U_i$ pairwise disjoint and non-empty by Lemma~\ref{l:disjointify} with possibly $k=0$.  Since $I\subseteq \ker \psi$, we conclude that $0=\psi(z)=\sum_{i=1}^k f_{a_i, U_i}$.  Since the $U_i$ are pairwise disjoint, we deduce that each $f_{a_i, U_i}=0$.  Since $\supp(a_i)=\ran(U_i)$, the only way $f_{a_i, U_i}$ vanishes is if $U_i=\emptyset$. We deduce that $k=0$ and hence $z\in I$.  This completes the proof.
\end{proof}


\begin{thebibliography}{10}

\bibitem{LeavittBook}
G.~{Abrams}, P.~{Ara}, and M.~{Siles Molina}.
\newblock {\em Leavitt path algebras}.
\newblock Number 2191 in Lecture Notes in Mathematics. London: Springer, 2017.

\bibitem{reconstruct}
P.~Ara, J.~Bosa, R.~Hazrat, and A.~Sims.
\newblock Reconstruction of graded groupoids from graded {S}teinberg algebras.
\newblock {\em Forum Math.}, 29(5):1023--1037, 2017.

\bibitem{Hazrat2017}
P.~Ara, R.~Hazrat, H.~Li, and A.~Sims.
\newblock Graded {S}teinberg algebras and their representations.
\newblock {\em Algebra Number Theory}, 12(1):131--172, 2018.

\bibitem{twistedreconstruction}
B.~Armstrong, G.~G. de~Castro, L.~O. Clark, K.~Courtney, Y.-F. Lin,
  K.~McCormick, J.~Ramagge, A.~Sims, and B.~Steinberg.
\newblock Reconstruction of twisted {S}teinberg algebras.
\newblock {\em arXiv e-prints}, page arXiv:2101.08556, Sept. 2021.

\bibitem{twists}
B.~{Armstrong}, L.~{Orloff Clark}, K.~{Courtney}, Y.-F. {Lin}, K.~{McCormick},
  and J.~{Ramagge}.
\newblock {Twisted Steinberg algebras}.
\newblock {\em arXiv e-prints}, page arXiv:1910.13005, Oct. 2019.

\bibitem{simpleskew}
V.~Beuter, D.~Gon\c{c}alves, J.~\"{O}inert, and D.~Royer.
\newblock Simplicity of skew inverse semigroup rings with applications to
  {S}teinberg algebras and topological dynamics.
\newblock {\em Forum Math.}, 31(3):543--562, 2019.

\bibitem{GonRoy2017a}
V.~M. Beuter and D.~Gon\c{c}alves.
\newblock The interplay between {S}teinberg algebras and skew rings.
\newblock {\em J. Algebra}, 497:337--362, 2018.

\bibitem{Bice2}
T.~{Bice}.
\newblock {Representing Rings on Ringoid Bundles}.
\newblock {\em arXiv e-prints}, page arXiv:2012.03006, Dec. 2020.

\bibitem{Bice1}
T.~{Bice}.
\newblock {Representing Semigroups on Etale Groupoid Bundles}.
\newblock {\em arXiv e-prints}, Oct. 2020.

\bibitem{operatorsimple1}
J.~Brown, L.~O. Clark, C.~Farthing, and A.~Sims.
\newblock Simplicity of algebras associated to \'etale groupoids.
\newblock {\em Semigroup Forum}, 88(2):433--452, 2014.

\bibitem{BCvH2017}
J.~H. Brown, L.~O. Clark, and A.~an~Huef.
\newblock Diagonal-preserving ring {$*$}-isomorphisms of {L}eavitt path
  algebras.
\newblock {\em J. Pure Appl. Algebra}, 221(10):2458--2481, 2017.

\bibitem{ExelBuss}
A.~Buss and R.~Exel.
\newblock Twisted actions and regular {F}ell bundles over inverse semigroups.
\newblock {\em Proc. Lond. Math. Soc. (3)}, 103(2):235--270, 2011.

\bibitem{CarlsenSteinberg}
T.~M. Carlsen and J.~Rout.
\newblock Diagonal-preserving graded isomorphisms of {S}teinberg algebras.
\newblock {\em Commun. Contemp. Math.}, 20(6):1750064, 25, 2018.

\bibitem{operatorguys2}
L.~Orloff~Clark and C.~Edie-Michell.
\newblock Uniqueness theorems for {S}teinberg algebras.
\newblock {\em Algebr. Represent. Theory}, 18(4):907--916, 2015.

\bibitem{Strongeffective}
L.~Orloff~Clark, C.~Edie-Michell, A.~an~Huef, and A.~Sims.
\newblock Ideals of {S}teinberg algebras of strongly effective groupoids, with
  applications to {L}eavitt path algebras.
\newblock {\em Trans. Amer. Math. Soc.}, 371(8):5461--5486, 2019.


\bibitem{ClarkPardoSteinberg}
L.~Orloff~Clark, R.~Exel, and E.~Pardo.
\newblock A generalized uniqueness theorem and the graded ideal structure of
  {S}teinberg algebras.
\newblock {\em Forum Math.}, 30(3):533--552, 2018.

\bibitem{nonhausdorffsimple}
L.~Orloff~Clark, R.~Exel, E.~Pardo, A.~Sims, and C.~Starling.
\newblock Simplicity of algebras associated to non-{H}ausdorff groupoids.
\newblock {\em Trans. Amer. Math. Soc.}, 372(5):3669--3712, 2019.

\bibitem{CenterLeavittGroupoid}
L.~Orloff~Clark, D.~Mart\'{\i}n~Barquero, C.~Mart\'{\i}n~Gonz\'{a}lez, and
  M.~Siles~Molina.
\newblock Using the {S}teinberg algebra model to determine the center of any
  {L}eavitt path algebra.
\newblock {\em Israel J. Math.}, 230(1):23--44, 2019.

\bibitem{GroupoidMorita}
L.~Orloff~Clark and A.~Sims.
\newblock Equivalent groupoids have {M}orita equivalent {S}teinberg algebras.
\newblock {\em J. Pure Appl. Algebra}, 219(6):2062--2075, 2015.

\bibitem{Demeneghi}
P.~Demeneghi.
\newblock The ideal structure of {S}teinberg algebras.
\newblock {\em Adv. Math.}, 352:777--835, 2019.

\bibitem{Donsig}
A.~P. Donsig, A.~H. Fuller, and D.~R. Pitts.
\newblock Von {N}eumann algebras and extensions of inverse semigroups.
\newblock {\em Proc. Edinb. Math. Soc. (2)}, 60(1):57--97, 2017.

\bibitem{Exel}
R.~Exel.
\newblock Inverse semigroups and combinatorial {$C\sp \ast$}-algebras.
\newblock {\em Bull. Braz. Math. Soc. (N.S.)}, 39(2):191--313, 2008.

\bibitem{exelrecon}
R.~{Exel}.
\newblock {Reconstructing a totally disconnected groupoid from its ample
  semigroup.}
\newblock {\em {Proc. Am. Math. Soc.}}, 138(8):2991--3001, 2010.

\bibitem{Skewasconv}
D.~{Gon{\c{c}}alves} and B.~{Steinberg}.
\newblock {{\'E}tale groupoid algebras with coefficients in a sheaf and skew
  inverse semigroup rings}.
\newblock {\em Canad. J. Math.}, to appear.

\bibitem{Kumjiandiagonal}
A.~Kumjian.
\newblock On {$C^\ast$}-diagonals.
\newblock {\em Canad. J. Math.}, 38(4):969--1008, 1986.

\bibitem{lausch}
H.~Lausch.
\newblock Cohomology of inverse semigroups.
\newblock {\em J. Algebra}, 35:273--303, 1975.

\bibitem{Lawson}
M.~V. Lawson.
\newblock {\em Inverse semigroups}.
\newblock World Scientific Publishing Co. Inc., River Edge, NJ, 1998.
\newblock The theory of partial symmetries.

\bibitem{LawsonLenz}
M.~V. Lawson and D.~H. Lenz.
\newblock Pseudogroups and their \'etale groupoids.
\newblock {\em Adv. Math.}, 244:117--170, 2013.

\bibitem{MatsMatu}
K.~Matsumoto and H.~Matui.
\newblock Continuous orbit equivalence of topological {M}arkov shifts and
  {C}untz-{K}rieger algebras.
\newblock {\em Kyoto J. Math.}, 54(4):863--877, 2014.

\bibitem{Nekrashevychgpd}
V.~Nekrashevych.
\newblock Growth of \'etale groupoids and simple algebras.
\newblock {\em Internat. J. Algebra Comput.}, 26(2):375--397, 2016.


\bibitem{Paterson}
A.~L.~T. Paterson.
\newblock {\em Groupoids, inverse semigroups, and their operator algebras},
  volume 170 of {\em Progress in Mathematics}.
\newblock Birkh{\"a}user Boston Inc., Boston, MA, 1999.

\bibitem{Renault}
J.~Renault.
\newblock {\em A groupoid approach to {$C\sp{\ast} $}-algebras}, volume 793 of
  {\em Lecture Notes in Mathematics}.
\newblock Springer, Berlin, 1980.

\bibitem{renaultcartan}
J.~Renault.
\newblock Cartan subalgebras in {$C^*$}-algebras.
\newblock {\em Irish Math. Soc. Bull.}, (61):29--63, 2008.

\bibitem{resendeetale}
P.~Resende.
\newblock {\'E}tale groupoids and their quantales.
\newblock {\em Adv. Math.}, 208(1):147--209, 2007.

\bibitem{groupoidapproachleavitt}
S.~W. {Rigby}.
\newblock {The groupoid approach to Leavitt path algebras}.
\newblock {\em ArXiv e-prints}, Nov. 2018.

\bibitem{mygroupoidalgebra}
B.~Steinberg.
\newblock A groupoid approach to discrete inverse semigroup algebras.
\newblock {\em Adv. Math.}, 223(2):689--727, 2010.

\bibitem{groupoidbundles}
B.~Steinberg.
\newblock Modules over \'etale groupoid algebras as sheaves.
\newblock {\em J. Aust. Math. Soc.}, 97(3):418--429, 2014.

\bibitem{groupoidprimitive}
B.~Steinberg.
\newblock Simplicity, primitivity and semiprimitivity of \'etale groupoid
  algebras with applications to inverse semigroup algebras.
\newblock {\em J. Pure Appl. Algebra}, 220(3):1035--1054, 2016.

\bibitem{mydiagonal}
B.~Steinberg.
\newblock Diagonal-preserving isomorphisms of \'{e}tale groupoid algebras.
\newblock {\em J. Algebra}, 518:412--439, 2019.

\bibitem{groupoidprime}
B.~Steinberg.
\newblock Prime \'{e}tale groupoid algebras with applications to inverse
  semigroup and {L}eavitt path algebras.
\newblock {\em J. Pure Appl. Algebra}, 223(6):2474--2488, 2019.

\bibitem{MyEffrosHahn}
B.~{Steinberg}.
\newblock {Ideals of \'etale groupoid algebras and Exel's Effros-Hahn
  conjecture}.
\newblock {\em J. Noncommut. Geom.}, to appear.

\bibitem{simplicity}
B.~{Steinberg} and N.~{Szak{\'a}cs}.
\newblock {Simplicity of inverse semigroup and {\'e}tale groupoid algebras}.
\newblock {\em Adv. Math.}, to appear.

\bibitem{wehrung}
F.~Wehrung.
\newblock {\em Refinement monoids, equidecomposability types, and {B}oolean
  inverse semigroups}, volume 2188 of {\em Lecture Notes in Mathematics}.
\newblock Springer, Cham, 2017.

\end{thebibliography}
\def\malce{\mathbin{\hbox{$\bigcirc$\rlap{\kern-7.75pt\raise0,50pt\hbox{${\tt
  m}$}}}}}\def\cprime{$'$} \def\cprime{$'$} \def\cprime{$'$} \def\cprime{$'$}
  \def\cprime{$'$} \def\cprime{$'$} \def\cprime{$'$} \def\cprime{$'$}
  \def\cprime{$'$} \def\cprime{$'$}

\end{document}